\documentclass[11pt, a4paper, reqno]{amsart}

\usepackage[colorinlistoftodos]{todonotes}

\usepackage[utf8]{inputenc}
\usepackage{lmodern}
\usepackage{amssymb,amsfonts,amsthm,amsmath,bbm,mathrsfs,aliascnt,mathtools,dsfont}
\usepackage[english]{babel}
\usepackage[colorlinks]{hyperref}

\usepackage{ulem}
\usepackage{nicefrac}

\usepackage{tikz} 
\usetikzlibrary{arrows.meta,decorations.markings}

\usepackage{tcolorbox} 

\usepackage{subcaption}

\usepackage{graphicx} %
\usepackage{sidecap, caption}
\usepackage{wrapfig}

\usepackage{multicol}
\usepackage{lipsum}
\usepackage{array}

\usepackage{enumitem}  

\newcommand{\W}{\mathbf{W}}
\newcommand{\D}{\mathcal{D}}

\newcommand{\conf}{\text{Conf}}
\newcommand{\pr}{\text{pr}}
\newcommand{\ball}{\text{B}}
\newcommand{\supp}{\text{supp}}

\renewcommand{\P}{\mathbb{P}}
\newcommand{\e}{\mathbf{e}}

\newcommand{\off}{\text{off}}
\newcommand{\on}{\text{on}}

\newcommand{\kk}{\mathbf{k}}

\newcommand{\id}{o}

\DeclareMathOperator{\len}{\mathrm{length}}
\theoremstyle{plain}

\newtheorem{thm}{Theorem}[section]
\newtheorem{cor}[thm]{Corollary}
\newtheorem{lemma}[thm]{Lemma}
\newtheorem{prop}[thm]{Proposition}
\newtheorem{rem}[thm]{Remark}

\newtheorem{remarks}[thm]{Remarks}

\theoremstyle{definition}
\newtheorem{definition}[thm]{Definition}

\newtheorem{ex}[thm]{Example}

\newcommand{\N}{\mathbb{N}}
\newcommand{\E}{\mathbb{E}}
\newcommand{\calI}{\mathcal{I}}

\theoremstyle{remark}
\newtheorem{remark}[thm]{Remark}

\begin{document}

\title[Cone-Additive Functions for Random Walks on Free Products]{Cone-Additive Functions for Random Walks on Free Products of Graphs}

\author[L. Gilch]{Lorenz A. Gilch}
\address{Lorenz A. Gilch,
University of Passau,
Faculty of Computer Science and Mathematics,
94030 Passau,
Germany
}
\email{Lorenz.Gilch@uni-passau.de}

\author[H. Oppelmayer]{Hanna Oppelmayer} %\orcid{0000-0002-1998-5407} 
\address{Hanna Oppelmayer,
Universität Innsbruck,
	Department of Mathematics.
	6020 Innsbruck, Austria.
}
\email{hanna.oppelmayer@uibk.ac.at}

\thanks{
This research was
funded in whole or in part by the \textbf{Austrian Science Fund (FWF)
10.55776/ESP4189024}.
}

%\date\today

\begin{abstract}  
We define cone-additive functions for random walks on free products of countable graphs. These functions satisfy a limit theorem under mild assumptions. In fact, cone-additivity is present in several well-studied notions, like entropy, asymptotic range and drift. Cone-additivity can be seen as a separation property by space -- a quite different perspective than the well-studied concept of sub-additivity in the context of free products of groups, which is a separation by time. In our inhomogeneous setting of free products of graphs, this separation by space allows us to deduce new limit theorems for travelling salesman problems (that is, distance functions of lamplighter random walks on free products), for weight functions on edges and the range of the $r$-th visit.
\end{abstract}

\maketitle

\tableofcontents

\section{Introduction}
\label{sec:introduction}

Let $V_1$ and $V_2$ be finite or countable, disjoint sets with $|V_i|\ge 2$, 
$i\in\{1,2\}$, and distinguished vertices $o_i\in V_i$. 
Their \textit{free product} $V = V_1 * V_2$ consists of all finite words
\[
x_1 x_2 \dots x_n
\]
formed from the alphabet 
$(V_1\setminus\{o_1\}) \cup (V_2\setminus\{o_2\})$, 
under the constraint that two consecutive letters $x_j,x_{j+1}$ never belong to the same factor 
$V_i\setminus\{o_i\}$. The empty word is denoted by $o$.
Furthermore, let $P_1$ and $P_2$ be transition matrices on $V_1$ and $V_2$, respectively. 
We consider a time-homogeneous, transient Markov chain $(X_n)_{n\in\mathbb N_0}$ on $V$ 
starting at $X_0=o$, whose transition matrix is obtained as a convex combination 
of suitably lifted versions of $P_1$ and $P_2$ (see Section~\ref{subsec:random-walks} 
for details). 
\par
For better visualisation, we may associate directed graphs 
$\mathcal X_1$, $\mathcal X_2$, and $\mathcal X$ with vertex sets 
$V_1$, $V_2$, and $V$, respectively, where there is an oriented edge from a vertex $x$ to a vertex $y$ 
if and only if the corresponding single-step transition probability is positive. On the graph $\mathcal{X}$ we define, for $x\in V$, the cone $C(x)$ rooted at $x$ in a natural way, which consists of all words/vertices in $V$ having prefix $x$. As we will see, the random walk converges to some infinite word, where a nested sequence of cones will be finally entered without further exits.
\par
In this article, we extract a property that many well-studied concepts, such as \textit{entropy}, \textit{drift}, \textit{range} and \textit{weight functions on edges}, share. 
Let $\Pi$ denote the set of finite paths in $\mathcal{X}$. Furthermore, let $f:\Pi\to\mathbb{R}$ be a function which is adapted to the cone structure, that is, if the random path $[X_0,\ldots,X_n]$ until time $n\in\N$ enters finally the (random) cones $C(y_1),\ldots,C(y_k)$, where $C(y_1)\supset C(y_2)\supset\ldots \supset C(y_k)$ and $y_j$ has word length $j$, then  $f$ can be evaluated piecewise along this nested sequence of cones which are finally entered. That is,
$$
f\bigl([X_0,\ldots,X_n]\bigr)= \sum_{j=1}^{k} f(\pi_j) + f(\pi^\ast_{n}) \quad \textrm{ almost surely,}
$$ 
where $\pi_1,\ldots,\pi_k,\pi_n^\ast$ are certain random path pieces, defined in Section~\ref{subsec: paths}, such that the vertices of $\pi_j$ are elements of $C(y_{j-1})\setminus \bigl( C(y_j)\setminus \{y_j\}\bigr)$.
Moreover, we assume that the value $f(\pi_j)$, $j\in\N$, remains unchanged when we cancel the common prefix $y_{j-1}$ out of each vertex in $\pi_j$ leading to a ``shifted'' version $y_{j-1}^{-1}\pi$ of $\pi$, that is,
$$
f(\pi_j)=f\bigl(y_{j-1}^{-1}\pi\bigr).
$$
We call this class of functions \textit{cone-additive}; see Definition \ref{def: cone add}.

The focus of this article lies on the study of the asymptotic behaviour of 
\[
\frac1n f\bigl([X_0,  \dots, X_n]\bigr)\quad \textrm{ as $n\to\infty$.}
\] 
Our main result is that cone-additivity gives rise to a limit theorem under a mild assumption.
We state here a simplified version using the growth. The general result is Theorem~\ref{thm:limit-theorem}.
\begin{thm}\label{thm: intro limit-theorem} 
Let $f:\Pi \to \mathbb{R}$ be a cone-additive function such that $|f(\pi)|$ grows at most polynomially in the length of $\pi\in \Pi$. Furthermore, assume that the power series $G(o,o|z)=\sum_{n\geq 0}\P[X_n=o]\cdot z^n$ has radius of convergence strictly bigger than $1$.
Then there exists some constant $\mathfrak{c}\in\mathbb{R}$ such that
$$
\lim_{n\to\infty} \frac1n f\bigl([X_0,\ldots,X_n]\bigr) = %\frac{\mathbb{E}_\varrho\bigl[f(\pi_1)]}{\mathbb{E}[\e_2-\e_1]}=\ell\cdot  %\mathbb{E}_\varrho\bigl[f(\pi_1)] =: 
\mathfrak{c}\quad \textrm{almost surely}.
$$    
Moreover, if $f$ is strictly positive  then $\mathfrak{c}>0$. %, see Corollary \ref{lem:c-not-zero}.
\end{thm} 
In Corollary~\ref{cor:c-formula} we provide an explicit description of the constant $\mathfrak{c}$.
Furthermore, we will demonstrate the power of our main result in many interesting examples, which form completely new results by themselves; see Section \ref{sec:applications}. This includes a generalised version of the asymptotic range for random walks on free products where we count the number of edges or vertices which are visited by the random walk exactly $r\in\N$ times; see Theorems \ref{thm: range} and \ref{thm:r-range}.
 Moreover, we will show that the rate of escape w.r.t. a distance function arising from weights/distances on its edges, exists; see Theorems \ref{thm: weights} and \ref{thm:weight-distance}. Furthermore, we study lamplighter random walks and show the existence of the rate of escape; see Theorem \ref{thm:LL}. 
\par
Random walks on free products have been investigated extensively in a large variety. We briefly summarise some of the main results. The asymptotic behaviour of return probabilities for random walks on free products has been analysed by several authors, amongst others Gerl and Woess~\cite{GerlWoess}, Woess~\cite{woess:86}, Sawyer~\cite{Sawyer}, Cartwright and Soardi~\cite{CartwrightSoardi}, Lalley~\cite{Lalley1,Lalley2}, and Candellero and Gilch~\cite{CandelleroG}. For free products of finite groups, Mairesse and Math\'eus~\cite{MairesseMatheus} derived explicit formulas for both the drift and the asymptotic entropy. In Gilch \cite{gilch:07,gilch:11}, alternative formulas for the drift as well as for the entropy of random walks on free products of graphs were obtained, and in Gilch \cite{gilch:22,gilch:24} the existence of the asymptotic range and the capacity of the range were proven. Moreover, Shi et al.~\cite{ShiEtAl} studied the spectral radius for random walks on certain classes of free products of graphs. 
\par
The significance of free products is closely related to Stallings' Splitting Theorem (see Stallings~\cite{Stallings}), which states that a finitely generated group $\Gamma$ has more than one (geometric) end if and only if $\Gamma$ admits a non-trivial decomposition as a free product with amalgamation or as an HNN extension over a finite subgroup. Both types of groups are treated in detail, for example, in Lyndon and Schupp~\cite{LyndonSchupp}. 
Recall that a free product is a special case of an amalgamated product over the trivial subgroup.
\par
Whereas most of the aforementioned works focus on random walks on free products of groups, which exhibit a highly homogeneous structure, the present article considers more general free products of graphs, characterised by a substantially less homogeneous structure, which requires additional techniques and ideas since the main tools from the group setting (e.g., applications of Kingman's Subadditive Ergodic Theorem) do not work any more. Thus, the present article extends the group-theoretic framework and leads to new results for non-group-invariant random walks.
\par
Our proofs rely heavily on generating function techniques and 
a detailed analysis of their interaction across the free product structure. These techniques for rewriting probability generating functions on the free product in terms
of functions on the single factors of the free product were introduced independently and
simultaneously by \cite{CartwrightSoardi}, \cite{woess:86}, Voiculescu \cite{voiculescu}, and McLaughlin \cite{McLaughlin}. 
\par 
The outline of the paper is as follows: in Section \ref{sec:free-products} we give an introduction to free products of graphs and equip them with a natural class of random walks. In particular, we state our Main Theorem  \ref{thm:limit-theorem} at the end of this section, which we will prove in  Section \ref{sec:proof-of-LLT}. 
In Section \ref{sec:applications}, we will demonstrate the power of our main theorem in various applications, leading to further completely new results.

\section{Free Products and Random Walks}
\label{sec:free-products}

\subsection{Free Products}\label{subsec:free products}

Suppose we are given finite or countable sets $V_1$ and $V_2$   with $|V_i|\geq 2$ for every $i\in\calI:=\{1,2\}$. W.l.o.g., we assume that $V_1$ and $V_2$ are disjoint; otherwise, we just rename the elements of $V_1$ and $V_2$. Since we are interested in transient random walks only, we exclude the case $|V_1|=|V_2|=2$, which will lead to recurrent random walks in our setting (see at the end of Subsection \ref{subsec:random-walks}). For each $i\in\calI$, we 
choose a distinguished element $o_i$ of $V_i$, which we call  \textit{root} of $V_i$. We set $V_i^\times:=V_i\setminus\{o_i\}$.
\par
Furthermore, suppose we are given a transition matrix $P_i=\bigl(p_i(x,y)\bigr)_{x,y\in V_i}$ on each $V_i$, $i\in\calI$, which gives rise to a (time-)homogeneous random walk on $V_i$.
For  $x, y\in V_i$, the associated $n$-step transition probabilities are denoted by $p_i^{(n)}(x, y)$. As we will see, only those elements  $x\in V_i$ will be of interest for us, which can be reached from $o_i$ with positive probability. Therefore, we may assume w.l.o.g. that, for every $i\in\calI$ and every $x\in V_i$, there exists some $n_x\in\N$ such that $p_i^{(n_x)}(o_i,x)>0$. Moreover, for the sake of simplicity and better readability of our proofs, we assume $p_i(x, x)= 0$ for every $i\in\calI$ and all $x\in V_i$; this assumption can be dropped without any restriction; see \cite[Section 6]{gilch:22}.
\par
For better visualization, we may equip $V_i$ with a graph structure w.r.t. $P_i$: we think of rooted graphs $\mathcal{X}_i$
with vertex sets $V_i$ and roots $o_i$ such that there is an oriented edge from $x\in V_i$ to $y\in V_i$ if and only if $p_i(x, y) > 0$.
\par
%For $i\in\calI$, set $V_i^\times:= V_i\setminus\{o_i\}$ and $V^\times_\ast := V_1^\times \cup V_2^\times$. 
The \textit{free product of $V_1$ and $V_2$} is given by the set
\begin{equation}\label{equ:free-product}
V:=V_1\ast V_2 := \left\lbrace v_1v_2\dots v_n \,\Biggl|\, \begin{array}{c} n\in\N, v_j\in V_1^\times \cup V_2^\times, \\
v_j\in V_k^\times \Rightarrow v_{j+1}\notin V_k^\times
\end{array}\right\rbrace \cup\{o\},
\end{equation}
the set of all finite words over the alphabet $V^\times_\ast:=V_1^\times \cup V_2^\times$ such that no two consecutive letters come from the same $V_i^\times$,
where $o$ describes the empty word. Observe that $V_i^\times\subset V$, and we may consider $o_i$  as the ``empty word'' of $V_i$ and also identify it with $o$. Throughout this article, we will use the representation of elements in $V$ as in (\ref{equ:free-product}).
\par
The \textit{type} $\delta(u)$ of $u=u_1\ldots u_m\in V\setminus\{o\}$ is defined to be $i\in\calI$ if $u_m\in V_i^\times$. %; we set $\delta(o) := 0$.
We have  a natural partial composition law for elements in $V$: if $u=u_1\dots  u_m, v = v_1\dots v_n \in V\setminus \{o\}$ with $\delta(u_m)\neq \delta(v_1)$, then $uv=u_1\ldots u_m v_1\ldots v_n\in V$ stands for the concatenation of the words $u$ and $v$, which is well-defined. Moreover, we set $uo_i := u$ for every $i\in\calI$ and $o_iu := u$ since $o_i$ is interpreted as the empty word in $V_i$ and $V$. %; in particular, $o_i$ is also identified with the empty word $o$. 
Since concatenation of words in $V$ is only partially defined, concatenation is \textit{not} a group operation on $V$; in particular, standard arguments from the group theory setting like Kingman's Subadditive Ergodic Theorem  can \textit{not} be applied directly; other approaches are necessary in order to derive the proposed results.  
\par
The set $V$ can also be equipped with a graph structure $\mathcal{X}$ which is constructed inductively as follows: take copies of $\mathcal{X}_1$ and $\mathcal{X}_2$ and glue them together at their roots $o_1$ and $o_2$ to one single common root, which becomes the empty word $o$; inductively, at each in the previous step newly-added vertex \mbox{$v=v_1\dots v_k\in V\setminus\{o\}$} with $v_k\in V_i$, attach a copy of $\mathcal{X}_j$, $j\in\calI \setminus\{i\}$, where $v$ is identified with $o_j$ from the new copy of $\mathcal{X}_j$, and the vertices $v_{k+1}\in V_j\setminus\{o_j\}$ from the new copy are now identified with $v_1\ldots v_k v_{k+1}$ within the free product. Then $\mathcal{X}$ is the \textit{free product of the graphs $\mathcal{X}_1$ and $\mathcal{X}_2$}; see e.g. \cite[Section 9.C]{woess}.

\begin{ex}\label{example:free-product}
Consider the sets $V_1=\{o_1,a\}$ and $V_2=\{o_2,b,c\}$ equipped with  the following graph structure:
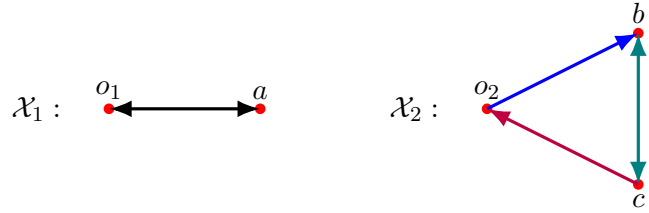
\begin{center}
\begin{tikzpicture}[scale=1]

\coordinate[label=left:$\mathcal{X}_1:$] (f) at (-0.5,0);
\coordinate[label=above:$o_1$] (e) at (0,0);
\coordinate[label=above:$a$] (a) at (2,0);

\coordinate[label=left:$\mathcal{X}_2:$] (g) at (4.5,0);
\coordinate[label=above:$o_2$] (e2) at (5,0);
\coordinate[label=above:$b$] (b) at (7,1);
\coordinate[label=below:$c$] (c) at (7,-1);

\fill[red] (e) circle (2pt);
\fill[red] (a) circle (2pt);
\fill[red] (b) circle (2pt);
\fill[red] (c) circle (2pt);
\fill[red] (e2) circle (2pt);
\draw[{Latex[length=3mm]}-{Latex[length=3mm]},very thick] (e) -- (a);
\draw[-{Latex[length=3mm]},very thick,blue] (e2) -- (b);
\draw[-{Latex[length=3mm]},very thick,purple] (c) -- (e2);
\draw[{Latex[length=3mm]}-{Latex[length=3mm]},very thick,teal] (b) -- (c);

\end{tikzpicture}
\end{center}
The graph $\mathcal{X}$ of the free product $V_1\ast V_2$ has then the following structure:
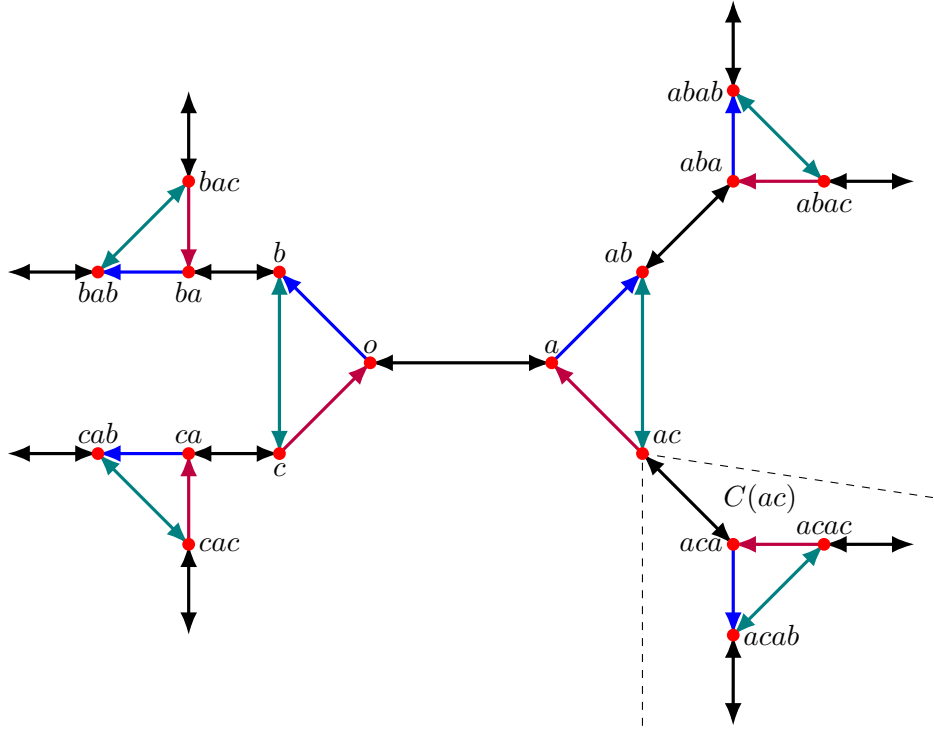
\begin{center}
\begin{tikzpicture}[scale=1.2]
\coordinate[label=above:$o$] (e) at (0,0);
\coordinate[label=above:$a$] (a) at (2,0);
\coordinate[label=170:$ab$] (ab) at (3,1);
\coordinate[label=10:$ac$] (ab2) at (3,-1);
\coordinate[label=170:$aba$] (aba) at (4,2);
\coordinate[label=left:$abab$] (abab) at (4,3);
\coordinate[label=below:$abac$] (abab2) at (5,2);
\coordinate[label=left:$aca$] (ab2a) at (4,-2);
\coordinate[label=right:$acab$] (ab2ab) at (4,-3);
\coordinate[label=above:$acac$] (ab2ab2) at (5,-2);

\coordinate[label=above:$b$] (b) at (-1,1);
\coordinate[label=below:$c$] (b2) at (-1,-1);

\coordinate[label=below:$ba$] (ba) at (-2,1);
\coordinate[label=below:$bab$] (bab) at (-3,1);
\coordinate[label=right:$bac$] (bab2) at (-2,2);

\coordinate[label=above:$ca$] (b2a) at (-2,-1);
\coordinate[label=above:$cab$] (b2ab) at (-3,-1);
\coordinate[label=right:$cac$] (b2ab2) at (-2,-2);

\draw[{Latex[length=3mm]}-{Latex[length=3mm]},very thick] (e) -- (a);
\draw[-{Latex[length=3mm]},very thick,blue] (e) -- (b);
\draw[{Latex[length=3mm]}-{Latex[length=3mm]},very thick,teal] (b) -- (b2);
\draw[-{Latex[length=3mm]},very thick,purple] (b2) -- (e);
\draw[-{Latex[length=3mm]},very thick,blue] (a) -- (ab);
\draw[{Latex[length=3mm]}-{Latex[length=3mm]},very thick,teal] (ab) -- (ab2);
\draw[-{Latex[length=3mm]},very thick,purple] (ab2) -- (a);
\draw[{Latex[length=3mm]}-{Latex[length=3mm]},very thick] (ab) -- (aba);
\draw[-{Latex[length=3mm]},very thick,blue] (aba) -- (abab);
\draw[{Latex[length=3mm]}-{Latex[length=3mm]},very thick,teal] (abab) -- (abab2);
\draw[-{Latex[length=3mm]},very thick,purple] (abab2) -- (aba);
\draw[{Latex[length=3mm]}-{Latex[length=3mm]},very thick] (abab) -- (4,4);
\draw[{Latex[length=3mm]}-{Latex[length=3mm]},very thick] (abab2) -- (6,2);
\draw[{Latex[length=3mm]}-{Latex[length=3mm]},very thick] (ab2) -- (ab2a);
\draw[-{Latex[length=3mm]},very thick,blue] (ab2a) -- (ab2ab);
\draw[{Latex[length=3mm]}-{Latex[length=3mm]},very thick,teal] (ab2ab) -- (ab2ab2);
\draw[-{Latex[length=3mm]},very thick,purple] (ab2ab2) -- (ab2a);
\draw[{Latex[length=3mm]}-{Latex[length=3mm]},very thick] (ab2ab2) -- (6,-2);
\draw[{Latex[length=3mm]}-{Latex[length=3mm]},very thick] (ab2ab) -- (4,-4);
\draw[{Latex[length=3mm]}-{Latex[length=3mm]},very thick] (b) -- (ba);
\draw[-{Latex[length=3mm]},very thick,blue] (ba) -- (bab);
\draw[{Latex[length=3mm]}-{Latex[length=3mm]},very thick,teal] (bab) -- (bab2);
\draw[-{Latex[length=3mm]},very thick,purple] (bab2) -- (ba);
\draw[{Latex[length=3mm]}-{Latex[length=3mm]},very thick] (bab2) --+ (0,1);
\draw[{Latex[length=3mm]}-{Latex[length=3mm]},very thick] (bab) --+ (-1,0);

\draw[{Latex[length=3mm]}-{Latex[length=3mm]},very thick] (b2) -- (b2a);
\draw[-{Latex[length=3mm]},very thick,blue] (b2a) -- (b2ab);
\draw[{Latex[length=3mm]}-{Latex[length=3mm]},very thick,teal] (b2ab) -- (b2ab2);
\draw[-{Latex[length=3mm]},very thick,purple] (b2ab2) -- (b2a);
\draw[{Latex[length=3mm]}-{Latex[length=3mm]},very thick] (b2ab2) --+ (0,-1);
\draw[{Latex[length=3mm]}-{Latex[length=3mm]},very thick] (b2ab) --+ (-1,0);

\draw[dashed] (3,-4) -- (3,-1) -- (6.3,-1.5);
\node at (4.3,-1.5) {$C(ac)$};

\fill[red] (e) circle (2pt);
\fill[red] (a) circle (2pt);
\fill[red] (b) circle (2pt);
\fill[red] (b2) circle (2pt);
\fill[red] (ab) circle (2pt);
\fill[red] (ab2) circle (2pt);
\fill[red] (aba) circle (2pt);
\fill[red] (abab) circle (2pt);
\fill[red] (abab2) circle (2pt);
\fill[red] (ab2a) circle (2pt);
\fill[red] (ab2ab) circle (2pt);
\fill[red] (ab2ab2) circle (2pt);
\fill[red] (ba) circle (2pt);
\fill[red] (bab) circle (2pt);
\fill[red] (bab2) circle (2pt);
\fill[red] (b2a) circle (2pt);
\fill[red] (b2ab) circle (2pt);
\fill[red] (b2ab2) circle (2pt);

\end{tikzpicture}
\end{center}
\end{ex}

The underlying graph structure of free products $V$ allows us to define paths on $V$.
A \textit{path} of length $n\in\N_0$ in $\mathcal{X}$ (or in $V$)  is  a sequence of vertices $[v_0,v_1,\dots,v_n]$ in $V$ such that there is an oriented edge in $\mathcal{X}$ from $v_{i-1}$ to $v_i$ for every $i\in\{1,\dots,n\}$. We remind that, for each $x\in V$, there is a path from $o$ to $x$ by construction of $\mathcal{X}$ and due to the assumption made at the beginning of this section. We denote by $\Pi$  the set of all \textit{finite paths} in $V$, while $\Pi_\infty$ consists of all \textit{infinite} paths $[v_0,v_1,\ldots]$ with $v_0,v_1,\ldots \in V$ such that there is an oriented edge between $v_i$ and $v_{i+1}$ for all $i\in\N_0$.
The \textit{word length} of a word $u = u_1\dots u_m\in V\setminus\{o\}$  is defined as $\Vert u\Vert:= m$. Additionally, we set $\Vert o\Vert := 0$. We denote by $\len(\pi)$  the path length of  $\pi\in \Pi$.
\par
The tree-like graph structure of free products motivates the following crucial definition: the \textit{cone} rooted at $x\in V\setminus \{o\}$ is given by the set
$$
C(x):=\bigl\lbrace y\in V \,\bigl|\, y \textrm{ has prefix } x\bigr\rbrace.
$$
In other words, $C(x)$ consists of all elements $y\in V$ such that each path from $o$ to $y$ has to pass through $x$; compare with Example \ref{example:free-product}, where the cone  $C(ac)$ contains  all elements $ac, aca, acab, acac,\ldots$ in between the dashed lines. Moreover, we set $C(o):=V$.
\par
Let be $v=v_1\ldots v_m\in V\setminus\{o\}$ and consider a path
$\pi=[u_0,u_1,\ldots,u_n]$ inside $C(v)$, that is, $u_i\in C(v)$ for each $i\in\{0,\ldots,n\}$. Then, for  $w\in V\setminus\{o\}$ with $\delta(w)\neq \delta(v_1)$, we denote by $w\pi$ the path arising from $\pi$ by shifting each element $u_i$ by $w$, that is,
$$
w\pi:=[wu_0,wu_1,\ldots,wu_n].
$$
Vice versa, if we write $u_i=vw_i$ then 
$$
v^{-1}\pi :=[w_0,w_1,\ldots,w_n]
$$
denotes the path  arising from $\pi$ by cancelling the common prefix $v$ in each vertex $u_i$. It is easy to check that both $w\pi$ and $v^{-1}\pi$ are again well-defined paths; alternatively, Corollary \ref{cor:shifted-paths} will give a proof after random walks on $V$ have been introduced. 
\par
We now want to define the restriction of paths onto cones. For this purpose, consider in the following a path $\pi=[u_0,u_1,\ldots,u_m]\in \Pi$ and $v\in V\setminus \{o\}$ such that $\{u_0,\ldots,u_m\}\cap C(v)\neq \emptyset$. Then we define the restriction of $\pi$ onto $C(v)$ by
$$
\pi\bigl|_{C(v)}:=[x_0,x_1,\ldots,x_\kappa],
$$
where
$$
x_0 := u_{i_0} \quad \textrm{ with } i_0:=\inf\bigl\{\ell\in\{0,\ldots,m\} \,\bigl|\, u_\ell \in C(v) \bigr\}
$$
and for $j\geq 1$: 
$$
x_j := u_{i_j} \textrm{ with } i_j:=\inf\bigl\{\ell\in\{1,\ldots,m\} \,\bigl|\, \ell > i_{j-1}, u_\ell \in C(v)\setminus\{x_{\ell-1}\} \bigr\}.
$$
That is, $\pi\bigl|_{C(v)}$ arises from $\pi$ by removing  all subpaths which are located \textit{outside} of $C(v)$, where successive occurrences of $v$ are eliminated. The latter is justified as follows: when $\pi$ exits $C(v)$, the last element of $\pi$ before the exit is $v$ and the first vertex of a possible re-entry is again $v$; this double occurrency of $v$ is eliminated, since the successive/double occurrence of $v$ arises from the re-entry induced by the subpath's edge in the exterior of $C(v)$ leading back to $v$. In particular, if $u_0\notin C(v)$ then we must have $x_0=v$. We illustrate this subpath construction in a picture later.
Finally, we remark that, if $\{u_0,\ldots,u_n\}\cap C(v)=\{v\}$, then $\pi\bigl|_{C(v)}=[v]$.
\par
Vice versa, we define the restriction of $\pi$ to the exterior of $C(v)$ as follows if $\{u_0,\ldots,u_m\}\cap C(v)^c\neq \emptyset$ then
$$
\pi\bigl|_{\neg C(v)}:=[w_0,w_1,\ldots,w_\lambda],
$$
where
$$
w_0 := u_{k_0} \quad \textrm{ with } k_0:=\inf\bigl\{\ell\in\{0,\ldots,m\} \,\bigl|\, u_\ell \notin C(v)\setminus\{v\} \bigr\}
$$
and for $j\geq 1$: 
$$
w_j := u_{k_j} \textrm{ with } k_j:=\inf\left\{\ell\in\{1,\ldots,m\} \,\biggl|\, \begin{array}{c} \ell > k_{j-1}, u_\ell \notin C(v)\setminus\{v\},\\ u_\ell\neq x_{\ell-1}\end{array} \right\}.
$$
In other words, $\pi\bigl|_{\neg C(v)}$ arises from $\pi$ by removing  all subpaths which are located \textit{inside} of $C(v)$, where once again successive occurrences of $v$ are eliminated. This is  justified due to the following:  when $\pi$ enters $C(v)$, the entry is at $v$, and the last element of $\pi$ before a (possible) exit of $C(v)$ is also $v$; this double occurrency of $v$ is eliminated, since the successive/double occurrence of $v$ arises from the return to $v$ induced by the corresponding subpath's edge in the interior of $C(v)$. Observe that, if \mbox{$\{u_0,\ldots,u_n\}\cap C(v)^c=\{v\}$,} then $\pi\bigl|_{\neg C(v)}=[v]$.
\par
We illustrate the subpath constructions of $\pi\bigl|_{C(v)}$ and $\pi\bigl|_{\neg C(v)}$ in the following picture:
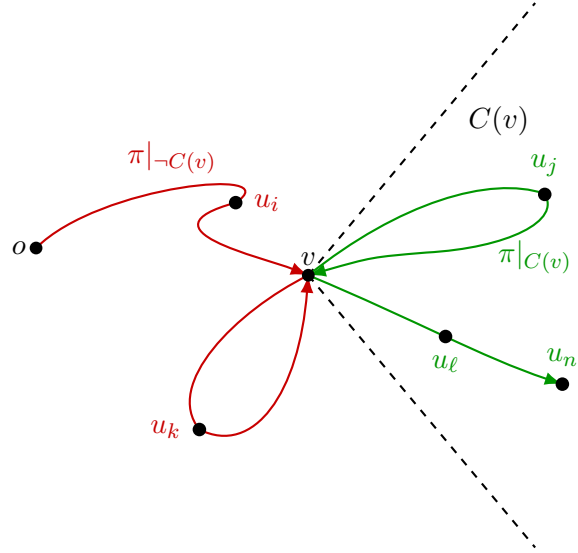
\begin{center}
\begin{tikzpicture}[>=Latex, scale=1.2]

% Farben
\definecolor{myred}{RGB}{200,0,0}
\definecolor{mygreen}{RGB}{0,150,0}

% Kegel
\draw[thick,dashed] (0,0) -- (2.5,3);
\draw[thick,dashed] (0,0) -- (2.5,-3);

% -------------------------
% Linke obere rote Kurve
% -------------------------
\draw[myred, thick, ->] 
(-3,0.3) 
.. controls (-2.3,1) and (-0.2,1.2) .. (-0.8,0.8)
.. controls (-1.9,0.5) and (-0.5,0.2) .. (0,0)
node[pos=0., circle, fill=black, inner sep=1.8pt, label=right:$u_i$] {};

% -------------------------
% Linke untere rote Loop
% -------------------------
\draw[myred, thick, ->]
(0,0) 
.. controls (-0.8,-0.4) and (-1.6,-1.2) .. (-1.2,-1.7)
.. controls (-0.6,-2.0) and (-0.1,-1.3) .. (0,0)
node[pos=0, circle, fill=black, inner sep=1.8pt, label=left:$u_k$] {};

% -------------------------
% Grüne obere Loop
% -------------------------
\draw[mygreen, thick, ->]
(0,0) 
.. controls (1.0,0.8) and (2.0,1.1) .. (2.6,0.9)
.. controls (2.8,0.6) and (2.2,0.3) .. (1.3,0.25)
node[pos=0.01, circle, fill=black, inner sep=1.8pt, label=above:$u_j$] {}
.. controls (0.7,0.2) .. (0,0);

% -------------------------
% Untere grüne Kurve
% -------------------------
\draw[mygreen, thick, ->]
(0,0) 
.. controls (1.2,-0.5) and (1.9,-0.9) .. (2.8,-1.2)
node[pos=0.5, circle, fill=black, inner sep=1.8pt, label=below:$u_\ell$] {}
node[pos=1.0, circle, fill=black, inner sep=1.8pt, label=above:$u_n$] {};

% Beschriftungen
\node at (2.1,1.7) {$C(v)$};
\node[mygreen] at (2.5,0.2) {$\pi|_{C(v)}$};
\node[myred] at (-1.5,1.3) {$\pi|_{\neg C(v)}$};
\fill (-3,0.3) circle (2pt);
\node[left] at (-3,0.3) {$o$};

% Punkt v
\fill (0,0) circle (2pt);
\node[above] at (0,0) {$v$};

\end{tikzpicture}
\end{center}
In the picture above, we consider the path 
$$
\pi=[o,\ldots,u_i,\ldots,v,\ldots,u_j,\ldots,v,\ldots,u_k,\ldots,v,\ldots, u_\ell,\ldots,u_n].
$$ 
Its restriction onto $C(v)$ is given by the concatenation of the  green paths
$$
\pi\bigl|_{C(v)}=[v,\ldots,u_j,\ldots,v,\ldots, u_\ell,\ldots,u_n],
$$
while  its restriction onto $C(v)^c$ is given by the concatenation of the red paths
$$
\pi\bigl|_{\neg C(v)}=[o,\ldots,u_i,\ldots,v,\ldots, u_k,\ldots,v].
$$
We have:
\begin{lemma}\label{lem:path-restriction-well-defined}
Let be $\pi=[u_0,\ldots,u_m]\in\Pi$ and $v\in V$. Then:
\begin{enumerate}
\item If $\{u_0,\ldots,u_m\}\cap C(v)\neq \emptyset$ then $\pi\bigl|_{C(v)}$ is a well-defined path.
\item If $\{u_0,\ldots,u_m\}\cap C(v)^c\neq \emptyset$ then $\pi\bigl|_{\neg C(v)}$ is a well-defined path.
\end{enumerate}
\end{lemma}
\begin{proof}
Let be $\pi=[u_0,\ldots,u_m]\in\Pi$ and $v\in V$ with $\{u_0,\ldots,u_m\}\cap C(v)\neq \emptyset$. Consider $\pi\bigl|_{C(v)}=[x_0,\ldots,x_\kappa]$.
Clearly, we either must have $x_0=u_0\in C(v)$ or the first element of $\pi\bigl|_{C(v)}$ has to be $x_{0}=v$ since the cone $C(v)$ can only be entered through $v$.
Assume now that the path $\pi$ has the form 
$$
[\ldots,u_{s-1},u_s=v,y_{1},\ldots,y_{t-1},y_t=u_{s+t}=v,u_{s+t+1},\ldots]
$$ 
with $y_1,\ldots,y_{t-1}\notin C(v)$. Then, by definition of $\pi\bigl|_{C(v)}$ the subpath $[y_1,\ldots,y_t]$ is removed and we get the ``reduced" path
$$
[\ldots,u_{s-1},u_s=v,u_{s+t+1},\ldots].
$$
If $u_{s+t+1}\in C(v)$ then we still have  that there exists an edge from $v=u_s=u_{t+s}$ to $u_{t+s+1}$; if $u_{s+t+1}\notin C(v)$, then we iterate the subpath removal procedure again. In any case, we have that two consecutive elements of $\pi\bigl|_{C(v)}$ are joined by an edge, since $u_\ell=u_{\ell+1}$ can only happen when $u_\ell=v$, but these double occurrences are removed by the construction of $\pi$'s restriction onto $C(v)$. This shows that $\pi\bigl|_{C(v)}$ is a well-defined path.
\par
Completely analogously, we can show that $\pi\bigl|_{\neg C(v)}$ is also a well-defined path if $\{u_0,\ldots,u_m\}\cap C(v)^c\neq \emptyset$.
\end{proof}

\begin{ex}
Consider Example \ref{example:free-product}. We have $a\in V_1^\times$ and $b\in V_2^\times$. If $\pi=[\id,a,ab,aba,ab,a,\id,b,\id,a]$, then 
$$
\pi\bigl|_{C(a)}=[a,ab,aba,ab,a],\ 
\pi\bigl|_{C(ab)}=[ab,aba,ab], \
\pi\bigl|_{C(b)}=[b],
$$
and
$$
\pi\bigl|_{\neg C(a)}=[o,a,o,b,o,a], \ 
\pi\bigl|_{\neg C(ab)}=[o,a,ab,a,o,b,o,a].
$$
\end{ex}

\subsection{Random Walks on Free Products}
\label{subsec:random-walks}

We now construct in a natural way a random walk on $V$ arising from $P_1$ and $P_2$. For this purpose, we lift the transition matrices $P_1$ and $P_2$ to transition matrices $\bar P_1=\bigl(\bar p_1(x,y)\bigr)_{x,y\in V}$ and $\bar P_2=\bigl(\bar p_2(x,y)\bigr)_{x,y\in V}$ on $V$: if $u\in V$ with $\delta(u)\neq i\in\calI$ and $v,w\in V_i$, then $\bar p_i(uv,vw):=p_i(v,w)$. Otherwise, we set $\bar p_i(x,y):=0$ for $x,y\in V$. Choose now any $\alpha\in (0,1)$ and fix it. We set $\alpha_1:=\alpha$ and $\alpha_2:=1-\alpha$. Then we define a new transition matrix $P$ on $V$  by
$$
P=\alpha \cdot \bar P_1 + (1-\alpha)\cdot \bar P_2=\bigl(p(x,y)\bigr)_{x,y\in V},
$$
which governs a time-homogeneous random walk $(X_n)_{n\in\N_0}$ on $V$ (or on $\mathcal{X}$). We set $X_0:=o$ as the starting point of our random walk. The random walk can be interpreted as follows: if the random walk stands at some vertex  $u=u_1\ldots u_m\in V$ with $\delta(u)=i\in\calI$,  a coin is tossed and afterwards -- in dependence of the outcome of the coin toss --  the random walk either performs one step within the copy of $\mathcal{X}_i$ to which $u$ belongs according to $\bar p_i(u_m,\cdot)$ \textit{or} one step is performed  into the new copy of $\mathcal{X}_j$, $j\in\calI\setminus\{i\}$, attached at $u$ according to $\bar p_j(o_j,\cdot)$. E.g., in the context of Example \ref{example:free-product}, if the random walk stands at $u=aca$, then it moves with probability $\alpha_1\cdot p_1(a,o_1)$ within the copy of $\mathcal{X}_1$ of $aca$ to $ac$ and with probability $\alpha_2\cdot p_2(o_2,b)$ to $acab$ within the next copy of $\mathcal{X}_2$ attached at $aca$. Observe that the graph $\mathcal{X}$ is the transition graph w.r.t. $P$.
\par
By definition, every path $[v_0,\dots,v_n]$ in $\mathcal{X}$ has strictly positive probability $\P\bigl[X_1=v_1,\dots,X_n=v_n|X_0=v_0\bigr]>0$ to be performed when starting at the path's initial vertex $v_0$.
As an abbreviation we write $\P_v[\,\cdot\,]:=\P[\,\cdot\, \mid X_0=v]$ for $v\in V$.
\par
The equation in the following lemma will be crucial in our proofs, which states that probabilities of paths within a cone depend only on their relative location to the cone's root:
\begin{lemma}\label{lem:cone-probs}
Let be $n\in\mathbb{N}$, $v\in V$ and $u_0,u_1,\dots,u_n\in C(v)$. Write $u_i=vu_i'$ for $i\in\{0,1,\dots,n\}$. Then:
$$
\P_{u_0}[X_1=u_1,\dots,X_n=u_n] = \P_{u_0'}[X_1=u'_1,\dots,X_n=u'_n].
$$
\end{lemma}
\begin{proof}
See \cite[Lemma 3.2]{gilch:22}.
\end{proof}
An immediate consequence of the last lemma is:
\begin{cor}\label{cor:shifted-paths}
Let be $v=v_1\ldots v_m\in V\setminus\{o\}$ and 
$\pi=[u_0,u_1,\ldots,u_n]\in \Pi$, $n\in\N$, with $u_0,\ldots,u_n\in C(v)$. Furthermore, let be $w\in V\setminus\{o\}$ with $\delta(w)\neq \delta(v_1)$. Then $w\pi$ and $v^{-1}\pi$ are well-defined paths.
\end{cor}
\begin{proof}
This follows immediately from Lemma \ref{lem:cone-probs} and from the fact that $w\pi$ and $v^{-1}\pi$ arise from $\pi$ by adding or removing some prefix at each element $u_i$, which preserves the adjacency in $\mathcal{X}$ between consecutive vertices of $\pi$.
\end{proof}

The structure of the free product together with Lemma \ref{lem:cone-probs} yields that, for all $i\in\mathcal{I}$ and all $v\in V\setminus\{o\}$ with $\delta(v)=i$, the probability 
\begin{equation*} %\label{equ:xi}
\xi_i:=\mathbb{P}\bigl[ \exists n\in\mathbb{N}: X_n\notin C(v)\mid X_0=v\bigr]>0
\end{equation*}
does not depend on $v$. In other words, 
$$
1-\xi_i = \mathbb{P}_v\bigl[ \forall t\geq 1: X_t\in C(v)\bigr]
$$ 
is the probability that the random walk starting at any $v\in V\setminus\{o\}$ with $\delta(v)=i$ does not leave the cone $C(v)$ any more.
In \cite[Lemma 2.3]{gilch:07} it has been proven that $\xi_i<1$ for every $i\in\mathcal{I}$.
\par
An important random walk's characteristic is given by the \textit{spectral radius at $o$}, which is defined as
$$
\varrho:= \limsup_{n\to\infty} \P[X_n=o]^{1/n}.
$$
As a \textit{basic main assumption} throughout this paper, we assume that 
$$
\varrho <1.
$$
This is equivalent to the fact that the \textit{Green function} 
$$
G(o,o|z):=\sum_{n\geq 0} p^{(n)}(o,o)\cdot z^n, z\in\mathbb{C},
$$
has radius of convergence $\mathscr{R}>1$.
This assumption implies \textit{transience} of the random walk governed by $P$ and excludes degenerate cases. In particular, this assumption excludes the irreducible, recurrent case $|V_1|=|V_2|=2$, in which $V$ becomes the free product $\bigl(\mathbb{Z}/(2\mathbb{Z})\bigr)\ast \bigl(\mathbb{Z}/(2\mathbb{Z})\bigr)$, where the underlying random walk is group-invariant and recurrent.
 If  $P_1$ or $P_2$ is not irreducible, then it is easy to check that $\varrho <1$. In general, this basic assumption is satisfied in many cases; e.g., if $P_1$ and $P_2$ govern irreducible and reversible random walks on $V_1$ and $V_2$, then $\varrho<1$; see Woess \cite[Theorem 10.3]{woess}. 

\subsection{Convergence of Random Walks}

\label{subsec:convergence-of-RW}

In this subsection, we summarise a few results on how the random walk $(X_n)_{n\in\N_0}$ introduced in the previous subsection converges in some sense to ``infinity''. For this purpose, we denote by $V_\infty$ the set of \textit{infinite} words $v_1v_2v_3\dots$ over the alphabet $V^\times_\ast$ such that no two consecutive letters arise from the same $V_i^\times$. For $u\in V$ and $v\in V_\infty$, we denote by $u\wedge v$ the common prefix of maximal length of $u=u_1\ldots u_m\in V$ and $v=v_1v_2\ldots$, that is,
$$
u \wedge v =u_1\ldots u_k, \textrm{ where }
k=\max\bigl\{l\in \{0,\ldots,m\} \mid u_1\ldots u_l=v_1\ldots v_l\bigr\}.
$$
In \cite[Proposition 2.5]{gilch:07} it has been shown that the random walk $(X_n)_{n\in\N_0}$ converges to some $V_\infty$-valued random variable $X_\infty$ in the sense that the length of the common prefix of $X_n$ and $X_\infty$ diverges to infinity almost surely. 
In other words, we have $\lim_{n\to\infty} \Vert X_n\wedge X_\infty\Vert=\infty$ almost surely. That is, the length of $X_n$ tends almost surely to infinity, and more and more letters at the beginning of $X_n$ finally stabilise.

\subsection{Final Cone Entry Times}\label{subsec: paths}

In this subsection, we introduce random times which are adapted to the limit process from the previous subsection.
Recall that $\Pi$ is the set of all finite paths in $V$. As explained in Subsection \ref{subsec:convergence-of-RW}, we have  $\Vert X_n\Vert\to\infty$ almost surely as $n\to\infty$. By the structure of free products, this implies that more and more letters at the beginning of $X_n$ stabilise, that is, those letters do \textit{not} change any more after some finite time. 
This motivates the definition of the \textit{$k$-th final cone entry time} for $k\in\mathbb{N}$:
\begin{eqnarray*}
\e_k &:=& \inf\Bigl\{ \ell\in\mathbb{N} \,\Bigl|\, \textrm{first $k$ letters of $X_\ell$ have stabilized} \Bigr\}\\
&=&\inf\Bigl\{ \ell\in\mathbb{N} \,\Bigl|\, \Vert X_\ell\Vert=k, \forall i\geq \ell:  X_i \in C(X_\ell) \Bigr\}.
\end{eqnarray*}
In other words, $\e_k$ is the first instant of time at which the random walk enters the cone $C(X_{\e_k})$ and remains afterwards in it for the entire future, that is, $\e_k$ is the first instant of time from which on the first $k$ letters do not change any more. In particular, we have $X_{\e_{k}-1}\notin C(X_{\e_k})$ and $C(X_{\e_k})\subset C(X_{\e_{k-1}})$. Additionally, we set $\e_0:=0$. These final cone entry times played a crucial role in several articles in the past;
compare, e.g.,  with Nagnibeda and Woess \cite{nagnibeda-woess} or \cite{gilch:07,gilch:22}. Since $\Vert X_n\Vert\to\infty$ almost surely, we have $\e_k<\infty$ almost surely for all $k\in\N$. Moreover, we set for $n\in\N$
$$
\mathbf{k}(n):= \max\bigl\{ k\in\N \mid \e_k\leq n\bigr\}.
$$
We introduce further notation associated with final cone entry times.
For $k\in \N$, the \textit{$k$-th sphere} is defined as
$$
\mathcal{S}_k:= \bigl(C(X_{\e_{k-1}}) \setminus C(X_{\e_{k}})\bigr)\cup \{X_{\e_{k}}\}.
$$
Since $X_{\e_{0}}:=\id$ and $C(\id)=V$ we have $\mathcal{S}_{1}= \bigl(V\setminus C(X_{\e_1})\bigr)\cup\{X_{\e_1}\}$.

For $k\in \mathbb{N}$, we define now the restriction of the random walk trajectory $[X_0,X_1,X_2,\ldots,]$  onto $\mathcal{S}_k$ by 
$$
\pi_k := [X_0,\ldots,X_{\e_k}]\bigl|_{\mathcal{S}_k} := \Bigl([X_0,\ldots,X_{\e_k}]\bigl|_{C(X_{\e_{k-1}})}\Bigr)\biggl|_{\neg C(X_{\e_{k}})}.
$$
In particular, we have $\pi_1=  [X_0,\ldots,X_{\e_1}]\bigl|_{\neg C(X_{\e_1})}$.
In other words, we remove from $[X_0,\ldots,X_{\e_k}]$ all vertices which do not belong to $\mathcal{S}_k$, including double successive occurrencies of the exit/entry vertices  of $\mathcal{S}_k$, namely $X_{\e_{k-1}}$ and $X_{\e_{k}}$. By Lemma \ref{lem:path-restriction-well-defined} and almost sure finiteness of $\e_{k}$, $\pi_k = [X_0,\ldots,X_{\e_k}]\bigl|_{\mathcal{S}_k}$ is a well-defined random path for each $k\in\N$ which goes from $X_{\e_{k-1}}$ to $X_{\e_k}$.
\par
Additionally, we set for $n\in\N$:
$$
\pi^\ast_{n}=\bigl[X_0,\ldots,X_n\bigr]\Bigl|_{C(X_{\e_{\kk(n)}})}.
$$

\begin{ex}
Consider Example \ref{example:free-product}, where 
$a\in V_1^\times$ and $b\in V_2^\times$, and the random walk trajectory 
$$
[X_0,X_1,\ldots]=[o,a,ab,aba,ab,a,o,b,o,a,ab,v_0,v_1\ldots]\in \Pi_\infty,
$$
where $v_0,v_1,\ldots \in C(ab)$. Then we have $\e_1=9$, $X_{\e_1}=a$ and $\e_2=10$, $X_{\e_2}=ab$. This leads to 
$$
\pi_1=[o,a,o,b,o,a],\quad  \pi_2=[a,ab,a,ab]
$$
and 
$$
\pi^\ast_{9}=[a,ab,aba,ab,a], \quad \pi^\ast_{10}=[ab,aba,ab].
$$
\end{ex}

\subsection{Cone-Additive Functions}

In the following, we introduce a class of functions which is adapted to the sphere decomposition principle of paths w.r.t. the last cone entry times.

\begin{definition}\label{def: cone add}
Consider a function 
$$
f: \Pi\to\mathbb{R}, \pi \mapsto f(\pi).
$$
Then $f$ is called  \textit{cone-additive} w.r.t. $P$ if the following two properties hold:
%for all finite paths $\pi=[\id=x_0,x_1,\ldots,x_n]\in\Pi$:
\begin{enumerate}[label=(\roman*)]
\item\label{cutoff} 
Additivity property: for all $n\in\N$,
$$
f\bigl([X_0,\ldots,X_n]\bigr)= \sum_{j=1}^{\mathbf{k}(n)} f(\pi_j) + f(\pi^\ast_{n}) \quad \textrm{ almost surely.}
$$
\item\label{shift invariance} Shift invariance: for all $j\in\N$, 
\begin{equation}\label{equ:shift-invariance-pi_k}
f(\pi_j)=f\bigl(X_{\e_{j-1}}^{-1}\pi_j\bigr).
\end{equation}
\end{enumerate}
\end{definition}
The additivity property of a function $f$ ensures that the (random) function value $f\bigl([X_0,\ldots,X_n]\bigr)$ can be evaluated by taking the function values of the subpaths $\pi_k$ located inside $\mathcal{S}_k$ plus some remnants at the  end of the path (that is, $f(\pi^\ast_{n})$). It can be seen as a separation property by space.
Moreover, cone-additivity ensures that the (random) function value $f(\pi_j)$, $j\in \N$,  does only depend on the relative location of $\pi_j$ within the cone $C(X_{\e_{j-1}})$ and the value of $f(\pi_j)$ does not change if we remove the common prefix $X_{\e_{j-1}}$ of the elements of $\pi_j$.
\par
In particular, shift invariance holds if $f(\pi)=f\bigl(v^{-1}\pi\bigr)$ for all $v\in V$ and $\pi=[u_0,\ldots,u_m]\in\Pi$  with $u_0,\ldots,u_m\in C(v)$. However, this may be more restrictive than necessary.

\begin{ex}
Consider the range of a path $\pi=[x_0,x_1,\ldots,x_n]\in \Pi$, $n\in\N$, which is given by
$$
\mathrm{Range}(\pi)=\{x_0,x_1,\ldots,x_n\}.
$$
We define the function $f:\Pi\to \N_0$ by
$$
\Pi \ni \pi'=[u_0,u_1,\ldots,u_m] \mapsto f(\pi'):=\bigl|\{u_0,\ldots,u_m\}\setminus \{u_0\}\bigr|.
$$
Then  $f(\pi_k)$ counts the number of distinct visited vertices of $(X_n)_{n\in\N_0}$ inside the set $\mathcal{S}_k':=\mathcal{S}_k\setminus \{X_{\e_{k-1}}\}$. Observe that $X_{\e_{k-1}}$ is already counted in $f(\pi_{k-1})$ if $k\geq 2$. Hence, $X_0=o$  is not counted in $f(\pi_1),\ldots, f(\pi_{\mathbf{k}(n)}), f(\pi_n^\ast)$ for any $n\in\mathbb{N}$. Since the sets $\bigl(\mathcal{S}_k'\bigr)_{k\in\N}$ are pairwise disjoint, $f$ fulfills the additivity property
$$
f\bigl([X_0,\ldots,X_n]\bigr)=f(\pi_1)+\ldots + f(\pi_{\mathbf{k}(n)}) + f(\pi_n^\ast).
$$
Shift invariance is obvious, that is, $f$
is cone-additive, and we have 
$$
f\bigl([X_0,\ldots,X_n]\bigr)=\bigl|\{X_0,\ldots,X_n\}\bigr|-1=\bigl|\{X_0,\ldots,X_n\}\setminus \{X_0\}\bigr|.
$$
\end{ex}
As we will see in Section \ref{sec:applications}, cone-additivity is satisfied in many further interesting cases.
\par
\begin{rem}
We defined the additivity property of cone-additive functions in terms of  the trajectories of the random walk $(X_n)_{n\in\N_0}$. However, it is also possible to define this property deterministically for finite paths. This would make it necessary to re-define last cone entry times for finite paths and to adapt the definition of the $\pi_k$'s. In this case, the values of  $\pi_k=\pi_k^{(n)}$ would depend on $n$ and could change as the random walk evolves, but these paths $\pi_k^{(n)}$ would finally stabilise after some finite time. In order to avoid superfluities and technical difficulties, we have agreed to the definition of the additivity property as in Definition \ref{def: cone add}.
\end{rem}

Let us now formulate our main result. 
We are interested in the asymptotic behaviour of $\frac1n f\bigl([X_0,\ldots,X_n]\bigr)$ as $n\to\infty$, where $f$ is  a  cone-additive function  on $\Pi$. In the following, denote by $\mathcal{R}>1$ the number from  \mbox{Proposition \ref{prop: E psi-k uniform finite}.}
Then:

\begin{thm} \label{thm:limit-theorem} 
Let $f:\Pi \to \mathbb{R}$ be a cone-additive function. 
Assume that $G(o,o|z)$ has radius of convergence strictly bigger than $1$ and  that there exist a constant $C\in (0,\infty)$ and $R_f \in (0,\mathcal{R})$ 
such that 
\begin{equation}\label{cond: rest bounded}
 \bigl|f\bigl(\pi\bigr)\bigr| \leq C \cdot R_f^{\mathrm{length}(\pi)}\quad \textrm{ for all } \pi\in\Pi.
\end{equation}
Then there exists some constant $\mathfrak{c}\in\mathbb{R}$ such that
$$
\lim_{n\to\infty} \frac1n f\bigl([X_0,\ldots,X_n]\bigr) = 
\mathfrak{c}\quad \textrm{almost surely}.
$$    
Moreover, if $f$ is strictly positive then $\mathfrak{c}>0$. 
\end{thm}

\begin{remarks}
\begin{enumerate}
\item Observe that Condition (\ref{cond: rest bounded}) is satisfied if, e.g., $f(\pi)$ grows at most polynomially in the length of $\pi$, that is, if there exists a polynomial $q(z)$ such that $\bigl|f(\pi)\bigr|\leq q\bigl(\mathrm{length}(\pi)\bigr)$ for all $\pi\in\Pi$. Hence, Theorem \ref{thm:limit-theorem} is a generalised version of Theorem \ref{thm: intro limit-theorem} in the Introduction.
\item We will present an explicit formula for $\mathfrak{c}$ in Corollary \ref{cor:c-formula}. From this formula follows also $\mathfrak{c}>0$ if $f$ is non-negative and if there exists a path $\pi=[o,u_1,\ldots,u_m]\in\Pi$ with $m\in\N$, $u_m\in V_i^\times$ for some $i\in\mathcal{I}$ and  $u_1,\ldots,u_{m-1}\in V \setminus \bigl(V_j^\times \cup C(u_m)\bigr)\cup\{u_m\}$, $j\in\mathcal{I}\setminus\{i\}$, such that $f\bigl([o,u_1,\ldots,u_m] \bigr)>0$.
\end{enumerate}
\end{remarks}

In Section \ref{sec:applications}, we will present several further theorems which follow as applications of Theorem \ref{thm:limit-theorem}. We will present different generalised drift and range theorems (see Theorems \ref{thm: range}, \ref{thm:r-range}, \ref{thm: weights} and \ref{thm:weight-distance}) as well as a local limit theorem w.r.t. the distance of lamplighter random walks on the free product (see Theorem \ref{thm:LL}). These theorems form new results by themselves and demonstrate the power of Theorem \ref{thm:limit-theorem}.

\section{Local Limit Theorem}
\label{sec:proof-of-LLT}

The aim of this section is to prove Theorem \ref{thm:limit-theorem}. For this purpose,
we construct a Markov chain which is adapted to the sequence $(\e_k)_{k\in \N_0}$ of final cone entry times and which captures the behaviour of $(X_n)_{n\in\N_0}$ within the spheres $\mathcal{S}_k$ and allows us to deduce the behaviour of $f(\pi_k)$. From this process, we will finally deduce the proposed Local Limit Theorem \ref{thm:limit-theorem}.

\subsection{Shift of Paths $\pi_k$}

In the following, we construct ``normalised'' versions of the paths $\pi_k$. For this purpose, we have to introduce further notation: if $X_{\e_k}=g_1\ldots g_k\in V$ for $k\in \N$, then we set
$$
\mathbf{W}_k=g_k,
$$
the last letter of $X_{\e_k}$. For $i\in\mathcal{I}$, denote by $V_i^\ast$ the set of all words $v\in V$ starting with a letter in $V_i^\times$ or being the empty word, that is,
$$
V_i^\ast =\bigl\{ g_1\ldots g_m\in V \,\bigl|\, m\in\N, g_1\in V_i^\times \bigr\} \cup \{o\}.
$$
Observe that $V_i^\ast$ can also be regarded as a cone $C_i=C_i(o):=V_i^\times$.  This allows us to define  
$$
\psi_1=[X_0,\ldots,X_{\e_1}]\Bigl|_{C_{\delta(X_{\e_1})}},
$$
which is obtained from $[X_0,\ldots,X_{\e_1}]$ by removing all subpaths which go into $V\setminus V^\ast_{\delta(X_{\e_1})}$ followed by cancellations of successive occurencies of $o$; that is, if $[X_0,\ldots,X_{\e_1}]=[u_0=o,u_1,\ldots,u_\kappa]$, then 
$$
\psi_1 = [o,y_1,\ldots,y_\lambda],
$$
where 
$$
y_0 := u_{i_0}=o \quad \textrm{ with } i_0:=0
$$
and for $j\geq 1$: 
$$
y_j := u_{i_j} \textrm{ with } i_j:=\inf\bigl\{\ell\in\{1,\ldots,\kappa\} \,\bigl|\, \ell > i_{j-1}, u_\ell \in V_{\delta(X_{\e_1})}^\ast\setminus\{y_{\ell-1}\} \bigr\}.
$$
Completely analogously to Lemma \ref{lem:path-restriction-well-defined}, one can show that $\psi_1$ is a well-defined random path. Observe that $\lambda$ is almost surely finite since $\e_1<\infty$ almost surely.
\par
For $k\geq 2$, we define
$$
\psi_k=X_{\e_{k-1}}^{-1} \Bigl([X_{0},\ldots,X_{\e_k}]\Bigl|_{C(X_{\e_{k-1}})}\Bigr).
$$
With this notation, we have the following link for $k\geq 2$:
$$
X_{\e_{k-1}}^{-1}\pi_k = \psi_k\bigl|_{\neg C(\mathbf{W}_k)},
$$
and shift invariance (\ref{equ:shift-invariance-pi_k}) implies
$$
f(\pi_k)=f\bigl(X_{\e_{k-1}}^{-1}\pi_k\bigr)=f\bigl(\psi_k\bigl|_{\neg C(\mathbf{W}_k)}\bigr)=f\bigl(g(\psi_k)\bigr),
$$
where $g(\mathbf{W}_k,\psi_k):=\psi_k\bigl|_{\neg C(\mathbf{W}_k)}$.
In the following, we show that $\psi_k$ is once again a well-defined path.

\begin{lemma} \label{lem:psi_k is a path}
For each $k\in\N$, $\psi_k$ takes values in $\Pi$. In particular, all vertices in $\psi_k$, which are different from $o$ start with a letter in $V_{\delta(X_{\e_k})}^\times$. Moreover, $\psi_k$ starts in $o$ and  ends with the letter $\mathbf{W}_k$.
\end{lemma}
\begin{proof}
Let $k$ be in $\N$ given. Since $\e_k<\infty$ almost surely, Lemma \ref{lem:path-restriction-well-defined} together with Lemma \ref{lem:cone-probs} ensure that $\psi_k$ is indeed a well-defined path.
\par
Since $\delta(X_{\e_{k-1}})\neq \delta(\mathbf{W}_k)$, all vertices in $\psi$ have to start with a letter in $V_{\delta(\mathbf{W}_k)}^\times=V_{\delta(X_{\e_k})}^\times$ or are equal to the empty word $o$. The first vertex of $\psi_k$ has to be $o$ since the first letter in $C(X_{\e_{k-1}})$ of $[X_0,\ldots,X_{\e_k}]$ is $X_{\e_{k-1}}$.
Finally,
by definition of $\psi_k$, the last letter of $\psi_k$ is given by
$$
X_{\e_{k-1}}^{-1} X_{\e_k}=\mathbf{W}_k.
$$
\end{proof}

\subsection{The Process $\bigl((\mathbf{W}_k,\psi_k)\bigr)_{k\in\N}$}
\label{subsec:process-Wk-psik}

The main goal in the following is to show that $\bigl((\mathbf{W}_k,\psi_k)\bigr)_{k\in\N}$ forms a homogeneous, irreducible, positive-recurrent Markov chain on the state space
$$
\mathcal{D}:=\mathrm{supp}\bigl((\mathbf{W}_1,\psi_1)\bigr).
$$
The first lemma shows that all random pairs $(\mathbf{W}_k,\psi_k)$, $k\in\N$, have indeed the same support.
\begin{lemma}\label{lem:support-D}
For all $k\in\N$,
$$
\mathrm{supp}\bigl((\mathbf{W}_k,\psi_k)\bigr)=\mathcal{D}.
$$
\end{lemma}
\begin{proof}
Let be $k\in\N$ with $k\geq 2$.
First, take any $(g,\varphi)\in \mathcal{D}$, where we assume w.l.o.g. that $g\in V_1^\times$ and $\varphi=[o,y_1,\ldots,y_n=g]\in\Pi$. Now choose any $g_0\in V$ with $\Vert g_0\Vert=k-1$ and $\delta(g_0)=2$. Furthermore, choose a shortest path from $o$ to $g_0$, say $[o,x_1,\ldots,x_m=g_0]$. This choice yields
$$
\{o,x_1,\ldots,x_m\}\cap C(g_0)=\{g_0\}.
$$
Now observe that $g_0y_i$, $i\in\{1,\ldots,n\}$, is well-defined since $\delta(g_0)=2$ and each $y_i$ is either the empty word or starts with a letter in $V_1$ by definition of $\psi_1$.
Then the concatenated path
$$
[o,x_1,\ldots,x_m=g_0,g_0y_1,\ldots,g_0y_n=g_0g]
$$
is a path from $o$ to $g_0g$ which allows to generate  \mbox{$(\mathbf{W}_k,\psi_k)=(g,\varphi)$} with positive probability: indeed,
\begin{eqnarray*}
&& \P\bigl[\mathbf{W}_k=g,\psi_k=\varphi\bigr] \\
&\geq & \P\left[\begin{array}{c}
X_1=x_1,\ldots, X_m=g_0,X_{m+1}=g_0y_1,\ldots,X_{m+n}=g_0g,\\
\forall t>m+n: X_t\in C(g_0g)
\end{array}\right]\\
&=& \P\left[ \begin{array}{c}
X_1=x_1,\ldots,X_m=g_0,\\
X_{m+1}=g_0y_1,\ldots X_{m+n}=g_0g \end{array}\right] \cdot
\P_{g_0g}\bigl[\forall t\geq 1: X_t\in C(g_0g)\bigr] \\
&=& \P\left[ \begin{array}{c}
X_1=x_1,\ldots,X_m=g_0,\\
X_{m+1}=g_0y_1,\ldots X_{m+n}=g_0g \end{array}\right] \cdot (1-\xi_1) >0.
\end{eqnarray*}
That is, we have shown that
$(g,\varphi)\in \mathrm{supp}\bigl((\mathbf{W}_k,\psi_k)\bigr)$. The case $g\in V_2^\times$ works completely analogously.
\par
Now let be $(g,\varphi)\in \mathrm{supp}\bigl((\mathbf{W}_k,\psi_k)\bigr)$ for some $k\geq 2$. By definition of $\psi_k$ and Lemma \ref{lem:psi_k is a path}, $\varphi=[o,y_1,\ldots,y_n=g]$ is a path from $o$ to $\mathbf{W}_k=g$. This yields:
\begin{eqnarray*}
&& \P\bigl[\mathbf{W}_1=g,\psi_1=\varphi\bigr] \\
&\geq & \P\bigl[
X_1=y_1,\ldots, X_n=y_n=g,
\forall t>n: X_t\in C(g)\bigr]\\
&\geq & \P\bigl[
X_1=y_1,\ldots, X_n=y_n=g\bigr] \cdot \P_g\bigl[
\forall t\geq 1: X_t\in C(g)\bigr]\\
&\geq & \P\bigl[
X_1=y_1,\ldots, X_n=y_n=g\bigr] \cdot (1-\xi_{\delta(g)})>0.
\end{eqnarray*}
That is, we have shown that $(g,\varphi)\in\mathcal{D}$. This finishes the proof.
\end{proof}

The next proposition will be a crucial element for the proof of our limit theorem.

\begin{prop}\label{prop: irred MC}
$\bigl((\mathbf{W}_k,\psi_k)\bigr)_{k\in\N}$ is a homogeneous, irreducible Markov chain on $\mathcal{D}$.
\end{prop}
\begin{proof}
First, we show the Markov property. For this purpose, let be $k\in\N_0$ and $(g_1,\varphi_1),\ldots,(g_{k+1},\varphi_{k+1})\in\mathcal{D}$ with 
$$
\P\Bigl[(\mathbf{W}_1,\psi_1)=(g_1,\varphi_1),\ldots, (\mathbf{W}_k,\psi_k)=(g_k,\varphi_k)\Bigr]>0.
$$
For $n\in\N$, we denote by $\Pi_n$ the set of all paths $[o,x_1,\ldots,x_n]\in\Pi$ such that
$$
\left[ \begin{array}{c}
X_1=x_1,\ldots,X_n=x_n,\\
\forall t>n: X_t\in C(x_n)
\end{array}
\right] \cap [\e_k=n] \cap\bigcap_{j=1}^k \bigl[\mathbf{W}_j=g_j,\psi_j=\varphi_j \bigr]\neq\emptyset,
$$
that is, the set of all paths which allow to generate $X_{\e_k}=g_1\ldots g_k$ at time $n$ such that $\mathbf{W}_j=g_j$ and $\psi_j=\varphi_j$ for all $j\in\{1,\ldots,k\}$. Moreover, for $m\in\N$, we write $\hat\Pi_m(\varphi_k,\varphi_{k+1},g_{k+1})$ for the set of all paths $[o,y_1,\ldots,y_m=g_{k+1}]\in \Pi$ such that
$$
\left[ \begin{array}{c}
\exists n\in\N \, \exists g_0\in V: \Vert g_0\Vert=k, X_{n-1}\notin C(g_0), \\
X_n=g_0,
X_{n+1}=g_0y_1,\ldots,X_{n+m}=g_0y_m,\\
\forall t>n+m: X_t\in C(g_0y_m)
\end{array} 
\right]\cap \left[\begin{array}{c}
 \e_{k+1}-\e_k=m,\\
\psi_k=\varphi_k,\\
\psi_{k+1}=\varphi_{k+1}\end{array}\right]
\neq \emptyset.
$$
With this notation, we obtain:
\begin{eqnarray*}
    && \P\Bigl[(\mathbf{W}_1,\psi_1)=(g_1,\varphi_1),\ldots, (\mathbf{W}_{k+1},\psi_{k+1})=(g_{k+1},\varphi_{k+1})\Bigr]\\
&=& \sum_{\substack{n\in\N,\\ [o,x_1,\ldots,x_n]\in \Pi_n}} \P\bigl[X_1=x_1,\ldots,X_n=x_n\bigr] \\
&&\quad \cdot \sum_{\substack{m\in\N,\\ [o,y_1,\ldots,y_m]\in \hat\Pi_m(\varphi_k,\varphi_{k+1},g_{k+1})}} \P\left[ \begin{array}{c} X_{n+1}=x_ny_1,\\
\vdots \\
X_{n+m}=x_ny_m \end{array} \middle|\, X_n=x_n\right]\\
&&\quad \cdot \P\Bigl[\forall t\geq 1: X_t\in C(x_ny_m)\,\Bigl|\, X_{n+m}=x_ny_m\Bigr] \\[1ex]
&\stackrel{\textrm{Lemma \ref{lem:cone-probs}}}{=}& \sum_{\substack{n\in\N,\\ [o,x_1,\ldots,x_n]\in \Pi_n}} \P\bigl[X_1=x_1,\ldots,X_n=x_n\bigr] \\
&&\quad \cdot \sum_{\substack{m\in\N,\\ [o,y_1,\ldots,y_m]\in \hat\Pi_m(\varphi_k,\varphi_{k+1},g_{k+1})}} \P\left[\begin{array}{c} X_{1}=y_1,\\ \vdots\\ X_{m}=y_m\end{array}\right] \cdot (1-\xi_{\delta(g_{k+1})}).
\end{eqnarray*}
Completely analogously, we obtain:
\begin{eqnarray*}
   && \P\Bigl[(\mathbf{W}_1,\psi_1)=(g_1,\varphi_1),\ldots, (\mathbf{W}_{k},\psi_{k})=(g_{k},\varphi_{k})\Bigr]\\
&=& \sum_{\substack{n\in\N,\\ [o,x_1,\ldots,x_n]\in \Pi_n}} \P\bigl[X_1=x_1,\ldots,X_n=x_n\bigr]\cdot (1-\xi_{\delta(g_{k})}).
\end{eqnarray*}
Therefore,
\begin{eqnarray}\label{eq: cond prob expressed with Xi}
   && \P\left[(\mathbf{W}_{k+1},\psi_{k+1})=(g_{k+1},\varphi_{k+1}) \,\middle| \, \begin{array}{c} (\mathbf{W}_1,\psi_1)=(g_1,\varphi_1),\\ \vdots \\ (\mathbf{W}_{k},\psi_{k})=(g_{k},\varphi_{k})\end{array}\right]
   \nonumber
   \\
&=& \frac{1-\xi_{\delta(g_{k+1})}}{1-\xi_{\delta(g_{k})}} \cdot 
    \sum_{\substack{m\in\N,\\ [o,y_1,\ldots,y_m]\in \hat\Pi_m(\varphi_k,\varphi_{k+1},g_{k+1})}} \P\left[ \begin{array}{c}
    X_{1}=y_1,\\
    \vdots \\
    X_{m}=y_m\end{array}\right].
    \label{equ:transition-probabilities-Wk-psik}
\end{eqnarray}
Since the right hand side depends only on $g_k,g_{k+1},\varphi_k$ and $\varphi_{k+1}$, we have shown that $\bigl((\mathbf{W}_k,\psi_k)\bigr)_{k\in\N}$ forms a homogeneous Markov chain.
\par
It remains to prove irreducibility of the Markov chain $\bigl((\mathbf{W}_k,\psi_k)\bigr)_{k\in\N}$. For this purpose, let be $(g_0,\varphi_0),(g_1,\varphi_1)\in \mathcal{D}$, and denote by $n$ the length of $\varphi_0$. Then the Markov chain $\bigl((\mathbf{W}_k,\psi_k)\bigr)_{k\in\N}$ starting at $(g_0,\varphi_0)$ can reach some state $(g,\varphi)\in\mathcal{D}$, where
$\delta(g)\neq \delta(g_1)$ and $\varphi=[o,g]$, 
with positive probability in at most $n+1$ steps; this can be achieved if $(X_n)_{n\in\N_0}$ starting at some vertex ending with letter $g_0$ moves away from $o$ by increasing the word length in each step. After having reached $(g,\varphi)$ the Markov chain $\bigl((\mathbf{W}_k,\psi_k)\bigr)_{k\in\N}$ can reach $(g_1,\varphi_1)$ in the next step if $(X_n)_{n\in\N_0}$ follows now the path $g\varphi_1$, which has positive probability to be performed. This shows that $\bigl((\mathbf{W}_k,\psi_k)\bigr)_{k\in\N}$ is irreducible.
\end{proof}

Recall that $\len(\pi)$ denotes the path length of  $\pi\in \Pi$. 
The following proposition will be an important key ingredient later concerning some required integrability property of the involved cone-additive functions.

\begin{prop}\label{prop: E psi-k uniform finite}
    There exists a real number $\mathcal{R}>1$ such that
    $$
    \sup_{k\in\N} \mathbb{E}\bigl[\mathcal{R}^{\mathrm{length}(\psi_k)}\bigr]<\infty.
    $$
\end{prop}
\begin{proof}
Define  for $i\in\mathcal{I}$, $y\in V_j$ with $j\in\mathcal{I}\setminus\{i\}$ and $n\in\N$
$$
k_i^{(n)}(o,y):=\P\Bigl[\forall \ell\in\{1,\ldots,n\}: X_\ell \notin V_i^\times, X_{n-1}\notin C(y), X_n=y\Bigr].
$$
and set
$$
\gamma_{i,j}(z):=\sum_{n\in\N, y\in V_j^\times} k_i^{(n)}(o,y)\cdot z^n, \ z\in\mathbb{C}.
$$
From \cite[Proof of Proposition 3.2]{gilch:07} follows that $\gamma_{i,j}(z)$ has radius of convergence strictly bigger than $1$.
\par
For better readability, we write, for $x\in V\setminus\{o\}$ with  $\delta(x)=i\in\mathcal{I}$,  $\neg\delta(x):=j$, where $j\in\mathcal{I}, j\neq i$.
\par
Now let be $k\in\N$ and denote by $T_{k-1}$ the random time of the first visit of $X_{\e_{k-1}}$. Then
$$
\mathrm{length}(\psi_k) \leq \e_k-T_{k-1} = (\e_k-\e_{k-1})+(\e_{k-1}-T_{k-1}).
$$
We now make a decomposition of all possible paths w.r.t. the values of $X_{\e_{k-1}}$, $X_{\e_{k}}$, $T_{k-1}$, $\e_{k-1}$ and $\e_k$. For this purpose, denote  by 
$$
D_{j}=\bigl\{x\in V\,\bigl|\, \Vert x\Vert =j\bigr\}
$$ 
the support of $X_{\e_{j}}$ for $j\in\N$. Then, for real $z>0$,
\begin{eqnarray*}
&& \mathbb{E}\bigl[z^{\mathrm{length}(\psi_k)}\bigr] \leq \mathbb{E}\bigl[z^{(\e_k-\e_{k-1})+(\e_{k-1}-T_{k-1})}\bigr] \\
&=& \sum_{\substack{m,n_1,n_2\in\N_0,\\ x\in D_{k-1},\\ y\in D_k\cap C(x),\\ x_1,\ldots, x_{n_1}\in V:\\ x_{n_1-1}\notin C(x),x_{n_1}=x,\\
y_1,\ldots,y_{n_2}\in C(x):\\ y_{n_2-1}\notin C(y), y_{n_2=y}}} 
\P\left[ \begin{array}{c} 
\forall j<m: x_j\notin C(x),\\
X_m=x, X_{m+1}=x_1,\\
\vdots \\ 
X_{m+n_1}=x_{n_1},\\
X_{m+n_1+1}=y_1,\\
\vdots \\ X_{m+n_1+n_2}=y_{n_2},\\
\forall k>m+n_1+n_2: X_k\in C(y)
\end{array}\right]\cdot z^{n_1+n_2}.
\end{eqnarray*}
Using the Markov property, we can decompose the above probabilities as follows:
\begin{eqnarray*}
&& \mathbb{E}\bigl[z^{\mathrm{length}(\psi_k)}\bigr]\\
&\leq & \sum_{\substack{m\in \N,\\ x\in D_{k-1}}} \P\Bigl[ X_m=x,\forall j<m: x_j\notin C(x)\Bigr]\\
&&\quad \cdot \sum_{\substack{n_1\in\N_0,\\ x_1,\ldots, x_{n_1}\in V:\\ x_{n_1-1}\notin C(x),x_{n_1}=x}} \P_x\Bigl[X_{m+1}=x_1,\ldots, X_{m+n_1}=x_{n_1}\Bigr]\cdot z^{n_1}\\
&&\quad \cdot \sum_{\substack{n_2\in\N,\\ y\in D_k\cap C(x),\\ y_1,\ldots,y_{n_2}\in C(x):\\ y_{n_2-1}\notin C(y), y_{n_2=y}}} \P_x\left[\begin{array}{c} X_{m+n_1+1}=y_1,\\ \vdots \\ X_{m+n_1+n_2}=y_{n_2}\end{array}\right]\cdot z^{n_2}\\
&&\quad \cdot \underbrace{\P_y\bigl[\forall t\geq 1: X_t\in C(y)\bigr]}_{=1-\xi_{\neg \delta(x)}}\\
&\stackrel{\textrm{Lemma \ref{lem:cone-probs}}}{\leq} & \underbrace{\sum_{\substack{m\in \N,\\ x\in D_{k-1}}} \P\Bigl[ X_m=x,\forall j<m: x_j\notin C(x)\Bigr]\cdot (1-\xi_{\delta(x)})}_{\leq \sum_{x\in D_{k-1}} \P[X_{\e_{k-1}}=x]=\P[\e_{k-1}<\infty]=1}\cdot \frac{1-\xi_{\neg\delta(x)}}{1-\xi_{\delta(x)}}\\
&&\quad \cdot \sum_{\substack{n_1\in\N_0,\\ x_1\ldots,x_{n-1}\in V}} \P_x\bigl[X_1=x_1,\ldots,X_{n-1}=x_{n-1},X_n=x]\cdot z^{n_1}\\
&&\quad \cdot \gamma_{\delta(x),\neg \delta(x)}(z) \\[1ex]
&\leq & \max\biggl\{\frac{1-\xi_1}{1-\xi_2},\frac{1-\xi_2}{1-\xi_1} \biggr\} \cdot \sup_{x\in V} G(x,x|z) \cdot \max\bigl\{ \gamma_{1,2}(z),\gamma_{2,1}(z)\bigr\}.
\end{eqnarray*}
Recall that $\xi_1,\xi_2\in (0,1)$.
By \cite[Lemma 3.6]{gilch:22}, there exists $R_0\in(1,\infty)$ such that 
$$
\sup_{x\in V} G(x,x|R_0)<\infty.
$$
By the remarks at the beginning of the proof, both $\gamma_{1,2}(z)$ and $\gamma_{2,1}(z)$ have also radii of convergence bigger than $1$. Therefore, there exists $\mathcal{R}>1$ independent of $k$ such that
$$
\mathbb{E}\bigl[\mathcal{R}^{\mathrm{length}(\psi_k)}\bigr] \leq \max_{i,j\in\mathcal{I}}\frac{1-\xi_i}{1-\xi_j} \cdot \sup_{x\in V} G(x,x|\mathcal{R}) \cdot \max\bigl\{ \gamma_{1,2}(\mathcal{R}),\gamma_{2,1}(\mathcal{R})\bigr\} <\infty,
$$
which proves the claim.
\end{proof}

An immediate consequence of the last proposition is that
$$
\sup_{k\in\N} \mathbb{E}\bigl[\mathrm{length}(\psi_k)\bigr]<\infty.
     $$
This statement also holds for conditional expectations:
\begin{cor}\label{cor: cond exp unif finite}
     For  every $(g,\varphi)\in \D$, 
     $$ 
     \sup_{k\in \N}\E\Bigl[\len(\psi_k) \,\Bigl|\,  (\W_1, \psi_1)=(g,\varphi)\Bigr] <\infty.
     $$ 
\end{cor}

\begin{proof}
Let be $(g,\varphi)\in \D=\mathrm{supp}\bigl((\W_1,\psi_1)\bigr)$.
Then:
\begin{eqnarray*}
&& \E\bigl[\len(\psi_k)\bigr] \\
&\geq & \P\bigl[(\W_1,\psi_1)=(g,\varphi)\bigr]\cdot 
\E\Bigl[\len(\psi_k) \,\Bigl|\,  (\W_1, \psi_1)=(g,\varphi)\Bigr]>0.
\end{eqnarray*}
Therefore, $\E\bigl[\len(\psi_k) \,\bigl|\,  (\W_1, \psi_1)=(g,\varphi)\bigr]$ must be bounded in $k$, because otherwise $\E\bigl[\len(\psi_k)\bigr]$ would be unbounded in $k$, a contradiction to Proposition \ref{prop: E psi-k uniform finite}. 
This proves the claim.  
\end{proof}
Now we can  prove:
\begin{prop} \label{prop:positive-recurrence}
The Markov chain
$\left((\mathbf{W}_k,\psi_k)\right)_{k\in\N}$ is positive-recurrent
\end{prop}
\begin{proof} 
By Proposition~\ref{prop: irred MC}, we know that $\left((\mathbf{W}_k,\psi_k)\right)_{k\in\N}$ is irreducible. Hence, it is sufficient to show positive recurrence of a single state in the support $\mathcal{D}$. For $(g,\varphi),(\widetilde{g}, \widetilde{\varphi})\in\mathcal{D}$ and $n\in\N$, let us write 
$$
q^{(n)}\bigl((g,\varphi), (\widetilde{g}, \widetilde{\varphi})\bigr):=\P\Bigl[(\W_{n+1},\psi_{n+1})= (\widetilde{g}, \widetilde{\varphi})\,\Big|\, (\W_1,\psi_1)=(g,\varphi)\Bigr] 
$$
for the $n$-step transition probabilities of $\bigl((\W_k,\psi_k)\bigr)_{k\in\N}$.
According to Feller \cite[Theorem p. 389]{feller} it suffices to  show that 
$$
\liminf\limits_{n\to\infty} q^{(n)}\bigl((g,\varphi), (g,\varphi)\bigr) >0
$$
for some $(g,\varphi)\in \mathcal{D}$.
Fix now any $g\in V_1^\times$ such that $p(\id, g)>0$, and  set $\varphi:=[\id, g] $. Note that $(g,\varphi)\in\mathcal{D} $. By Corollary~\ref{cor: cond exp unif finite}, we may choose $M\in 2\mathbb{N}$ such that 
$$ 
\forall k\in\mathbb{N}:\ \E\Bigl[\len(\psi_k) \,\Bigl|\,  (\W_1, \psi_1)=(g,\varphi)\Bigr] \leq M.
$$
This implies that there exists $\delta>0$ such that 
$$
%p_k:=
\mathbb{P}\Bigl[\len(\psi_k) \leq M +1 \,\Bigl|\,  (\W_1, \psi_1)=(g,\varphi)\Bigr] \geq \delta,
$$ for every $k\in\mathbb{N}$.
Indeed, if the above inequality would not hold, 
then for every $m\in\N$ there would exist  $k_m\in\N$ such that 
$$
\mathbb{P}\Bigl[\len(\psi_{k_m}) \leq M +1 \,\Bigl|\,  (\W_1, \psi_1)=(g,\varphi)\Bigr] \leq \frac{1}{m}.
$$
Thus, 
$$
\P\Bigl[\len(\psi_{k_m}) > M +1 \,\Bigl|\,  (\W_1, \psi_1)=(g,\varphi)\Bigr] \geq 1- \frac{1}{m}.
$$
But then
\begin{eqnarray*}
  M  &\geq &  \E \Bigl[\len(\psi_{k_m}) \,\Bigl|\,  (\W_1, \psi_1)=(g,\varphi)\Bigr]
  \\
  &=& \sum_{L>M+1} L \cdot\P \Bigl[\len(\psi_{k_m}) =L  \,\Bigl|\,  (\W_1, \psi_1)=(g,\varphi)\Bigr]
  \\
  &&+
  \sum_{L\leq M+1} L \cdot\P \Bigl[\len(\psi_{k_m}) =L  \,\Bigl|\,  (\W_1, \psi_1)=(g,\varphi)\Bigr]
  \\
  &\geq  & \sum_{L>M+1} L\cdot \P \Bigl[\len(\psi_{k_m}) =L  \,\Bigl|\,  (\W_1, \psi_1)=(g,\varphi)\Bigr]
  \\
   &\geq & (M+1)\cdot \sum_{L>M+1} \P \Bigl[\len(\psi_{k_m}) =L  \,\Bigl|\,  (\W_1, \psi_1)=(g,\varphi)\Bigr]
  \\ 
  &\geq & (M+1) \cdot \P \Bigl[\len(\psi_{k_m}) >M +1 \,\Bigl|\,  (\W_1, \psi_1)=(g,\varphi)\Bigr]
  \\ 
  &\geq & \left(M+1\right) \cdot \left(1-\frac{1}{m} \right),
\end{eqnarray*}
a contradiction for $m$ large enough.
Thus, we have
\begin{eqnarray}
    \delta &\leq &
\mathbb{P}\Bigl[\len(\psi_{k+1}) \leq M +1 \,\Bigl|\,  (\W_1, \psi_1)=(g,\varphi)\Bigr]
\nonumber
\\ 
&=& 
\sum_{\substack{(\widetilde{g}, \widetilde{\varphi})\in\mathcal{D}:
\\ \len(\widetilde{\varphi})\leq M+1}}
q^{(k)}\bigl((g,\varphi), (\widetilde{g}, \widetilde{\varphi})\bigr)
\label{equ:q-sum-bound}
\end{eqnarray} 
for all $k\in\mathbb{N}$.

The next step is to  show that there exists $\epsilon>0$ such that 
\begin{equation}\label{eq: qM bounded from below}
     q^{(M)}\bigl( (\widetilde{g}, \widetilde{\varphi}), (g,\varphi)\bigr) \geq \epsilon
\end{equation}
for all $(\widetilde{g}, \widetilde{\varphi})\in\mathcal{D}$ with $\len(\widetilde{\varphi})\leq M+1$.
Once we have established Inequality (\ref{eq: qM bounded from below}), the proof concludes as follows for $n>M$:
\begin{eqnarray*}
&& q^{(n)}\bigl((g,\varphi), (g,\varphi)\bigr)\\
 &=&
 \sum_{ (\widetilde{g}, \widetilde{\varphi})\in\mathcal{D}}
 q^{(n-M)}\bigl((g,\varphi), (\widetilde{g}, \widetilde{\varphi})\bigr) \cdot
 q^{(M)}\bigl( (\widetilde{g}, \widetilde{\varphi}), (g,\varphi)\bigr)
    \\
    &\geq &
    \sum_{\substack{(\widetilde{g}, \widetilde{\varphi})\in\mathcal{D}:
\\ \len(\widetilde{\varphi})\leq M+1}}
 q^{(n-M)}\bigl((g,\varphi), (\widetilde{g}, \widetilde{\varphi})\bigr) \cdot
 q^{(M)}\bigl( (\widetilde{g}, \widetilde{\varphi}), (g,\varphi)\bigr)
 \\
 &\geq &  \sum_{\substack{(\widetilde{g}, \widetilde{\varphi})\in\mathcal{D}:
\\ \len(\widetilde{\varphi})\leq M+1}}
 q^{(n-M)}\bigl((g,\varphi), (\widetilde{g}, \widetilde{\varphi})\bigr) \cdot
\epsilon%^M
\\ 
&\stackrel{(\ref{equ:q-sum-bound})}{\geq} & \delta \cdot \epsilon%^M
>0,
\end{eqnarray*}
which implies for $n>M$ that  $q_{(g,\varphi), (g,\varphi)}^{(n)}$ is bounded away from zero, yielding positive recurrence of $(g,\varphi)$.
\\
Hence, it is left to verify Inequality~(\ref{eq: qM bounded from below}). Now let $(\widetilde{g}, \widetilde{\varphi})\in\mathcal{D}$ be with $\len(\widetilde{\varphi})\leq M+1 $. W.l.o.g. we assume that $\widetilde{g}\in V_1^{\times} $; the  case $\widetilde{g}\in V_2^{\times} $ follows analogously by  exchanging the roles of $V_1$ and $V_2$. 
We will construct a sequence of states which leads from $(\widetilde{g}, \widetilde{\varphi}) $ to $ (g,\varphi)$ in $M$ steps with probability of at least $\epsilon.$
Fix now some arbitrary $g'\in V_1^{\times}$ and $ g''\in V_2^{\times}$ with $p(\id, g')>0$ and $p(\id, g'')>0$. Set 
$$
\epsilon:=\sqrt[M]{\min\bigl\{p(\id, g')\, , \,  p(\id,g'')\bigr\} }>0.
$$ 
Furthermore, we set $g_0:=\widetilde{g}$ and then alternating $$
g_1:= g'', \ g_2:=g', \ g_3:=g'', \ g_4:=g', \ldots, g_{M-2}:= g', \ g_{M-1}:=g''.%, \textrm{ and } g_M:=g.
$$
Moreover, we define  
$\varphi^{(0)}:= \widetilde{\varphi}$, which is a path from $o$ to $\widetilde{g}$, and inductively $\varphi^{(i)}$ for $i=1,\ldots, M-1$ as follows: 
$$
\varphi^{(i)}:=\begin{cases}
[o,g_i],& \textrm{if $\varphi^{(i-1)}$ does \textit{not} enter $C(g_{i-1})\setminus\{g_{i-1}\}$,}\\
[u_0,\ldots,u_\lambda=o,g_i],& \textrm{if
$g_{i-1}^{-1}\left( \varphi^{(i-1)}\bigl|_{C(g_{i-1})}\right)=[u_0,\ldots,u_\lambda=o]$.}
\end{cases}
$$
Observe that in the second case $\varphi^{(i-1)}\bigl|_{C(g_{i-1})}$ must be a path from $g_{i-1}$ to $g_{i-1}$ within $V^\ast_{\delta(g_i)}$, and therefore we have $u_0=u_\lambda=o$ such that $\varphi^{(i)}$ is indeed a well-defined path from $o$ to $g_i$ within $V_{\delta(g_i)}^\ast$. 
Note also that \mbox{$(g_i,\varphi^{(i)})\in \mathcal{D}$:} if the Markov chain $\bigl((\W_k,\psi_k)\bigr)_{k\in\N}$ is at state $(g_{i-1},\varphi^{(i-1)})$, then the underlying random walk $(X_n)_{n\in\N_0}$ can walk from some $X_{\e_{i-1}}$ ending with $g_{i-1}$ in one single step to $X_{\e_{i}}=X_{\e_{i-1}}g_i$ and stay within $C(X_{\e_{i}})$ afterwards, which leads to $(\W_i,\psi_i)=(g_i,\varphi^{(i)})$. In particular, 
the paths $\varphi^{(i)}$ are getting shorter and shorter as $i$ grows until it finally becomes $[o,g_i]$ for some $i\leq M-1$, and we set 
$$
\varphi^{(M)}:=\varphi=[o,g]\quad \textrm{ and }\quad g_M:=g.
$$
Next, we show that $(g_0,\varphi^{(0)}),\ldots,(g_M,\varphi^{(M)})$ define indeed a sequence of $M$ consecutive states of the Markov chain $\bigl((\W_k,\psi_k)\bigr)_{k\in\N}$ with positive transition probabilities.
 The Markov property together with  Equation~(\ref{eq: cond prob expressed with Xi}) in the proof of Proposition~\ref{prop: irred MC} yields

\begin{eqnarray*}
&& q^{(1)}\bigl((g_{i-1}, \varphi^{(i-1)}) , (g_i, \varphi^{(i)}) \bigr)
\\[1ex]
&=&
\P\Bigl[  (\W_i,\psi_i)=(g_i, \varphi^{(i)}) \,\Bigl|\,  (\W_{i-1},\psi_{i-1})=(g_{i-1}, \varphi^{(i-1)})  \Bigr]
\\
    &=& 
    \frac{1-\xi_{\delta(g_{i})}}{1-\xi_{\delta(g_{i-1})}} \cdot 
    \sum_{\substack{m\in\N,\\ [o,y_1,\ldots,y_m]\in \hat\Pi_m(\varphi^{(i-1)},\varphi^{(i)},g_{i})}} \P\bigl[X_{1}=y_1,\ldots,X_{m}=y_m\bigr]
    \\    
    &\geq &   \frac{1-\xi_{\delta(g_{i})}}{1-\xi_{\delta(g_{i-1})}} \cdot  p(\id, g_i)>0,
\end{eqnarray*} 
for all $i=1,\ldots, M$. Here, $\hat\Pi_m(\varphi^{(i-1)},\varphi^{(i)},g_{i})$ is defined as in the proof of Proposition~\ref{prop: irred MC}.
Thus, using the above inequality, we see that 
\begin{eqnarray*}
&&     q^{(M)}\bigl((\widetilde{g}, \widetilde{\varphi}), (g,\varphi)\bigr) \\
&\geq &
     \prod_{i=1}^M q^{(1)}\bigl((g_{i-1},\varphi^{(i-1)}), (g_i,\varphi^{(i)})\bigr)
     \\ 
     &=&
      \prod_{i=1}^{\frac{M}{2}} q^{(1)}\bigl((g_{2i-2},\varphi^{(2i-2)}), (g_{2i-1},\varphi^{(2i-1)})\bigr)
      \cdot 
      q^{(1)}\bigl((g_{2i-1},\varphi^{(2i-1)}), (g_{2i},\varphi^{(2i)})\bigr)
      \\ 
      &\stackrel{(\ast)}{\geq} &
       \prod_{i=1}^{\frac{M}{2}} \frac{1-\xi_{\delta(g_{2i-1})}}{1-\xi_{\delta(g_{2i-2})}} \cdot  p(\id, g_{2i-1})
       \cdot 
       \frac{1-\xi_{\delta(g_{2i})}}{1-\xi_{\delta(g_{2i-1})}} \cdot  p(\id, g_{2i})
       \\ 
       &= &
        \prod_{i=1}^{\frac{M}{2}}  p(\id, g_{2i-1})
        \cdot  p(\id, g_{2i}) 
        \geq  \epsilon,
\end{eqnarray*}
where we have used that $\delta(g_{2i}) = \delta(g_{2i-2}) $ in Inequality $(\ast)$. This proves Inequality~(\ref{eq: qM bounded from below}) and thus finishes the proof.
\end{proof}

\subsection{Proof of the Local Limit Theorem}

Recall that, by shift-invariance (\ref{equ:shift-invariance-pi_k}) of cone-additive functions, we have for every $k\geq 2$:
$$
f(\pi_k)=f\bigl(X_{\e_{k-1}}^{-1}\pi_k\bigr)
=f\Bigl(\psi_k\bigl|_{\neg C(\mathbf{W}_k)}\Bigr).
$$
The additivity property then yields
\begin{eqnarray}
&& f\bigl([X_0,\ldots,X_n]\bigr)\nonumber\\
&=& f(\pi_1) + \sum_{i=2}^{\mathbf{k}(n)} f(\pi_i)+ f\bigl(\pi^\ast_{n}\bigr)\nonumber\\
&=& f(\pi_1) + \sum_{i=2}^{\mathbf{k}(n)} f\Bigl(\psi_i\bigl|_{\neg C(\mathbf{W}_i)}\Bigr)+ f\bigl(\pi^\ast_{n}\bigr).\label{equ:f-decomposition}
\end{eqnarray}

The next lemma shows that $f(\pi_1)$ does not play a role in the asymptotic behaviour of $\frac1n f\bigl([X_0,\ldots,X_n]\bigr)$:

\begin{lemma}\label{lem: beginning to 0}
$$
\lim_{n\to\infty} \frac{f(\pi_1)}{n}=0 \quad \textrm{almost surely.}
$$    
\end{lemma}
\begin{proof}
    By transitivity of the random walk, we have that $\e_1<\infty$ almost surely.
    Since $\pi_1=[X_0,\ldots,X_{\e_1}]\bigl|_{\neg C(\mathbf{W}_1)}$, the value $f(\pi_1)$ is almost surely finite. Letting $n\to\infty$ yields the claim.
\end{proof}

Recall that  $\mathcal{D}$ denotes the support of $(\mathbf{W}_1,\psi_1)$ (and thus, the support of any $(\mathbf{W}_k,\psi_k)$, $k\in\N$). In the following, we introduce  the definition of  expectations w.r.t. some invariant distributions. 
Denote by $\varrho$ the unique invariant probability measure of the $\mathcal{D}$-valued Markov chain $\bigl((\mathbf{W}_k,\psi_k)\bigr)_{k\in\N}$ which exists due to positive recurrence. Let $h:\mathcal{D}\to\mathbb{R}$ be a function. For a $\mathcal{D}$-valued random vector $Z$,  we define the following expectation w.r.t. $\varrho$ as
$$
\mathbb{E}_\varrho\bigl[h(Z)\bigr] := \int h(Z)\,d\varrho= \sum_{(g,\varphi)\in\mathcal{D}} \varrho(g,\varphi)\cdot h(g,\varphi),
$$
provided the expectation exists, that is, if
$$
\mathbb{E}_\varrho\bigl[\bigl|h(Z)\bigr|\bigr] := \sum_{(g,\varphi)\in\mathcal{D}} \varrho(g,\varphi)\cdot \bigl|h(g,\varphi)\bigr|<\infty.
$$
Recall that $\mathcal{R}$ is the number from Proposition \ref{prop: E psi-k uniform finite}. In particular, we write for $R\in(0,\mathcal{R})$
$$
\mathbb{E}_\varrho\bigl[R^{\len(\psi_1)}\bigr] = \int R^{\len(\psi_1)}\,d\varrho= \sum_{(g,\varphi)\in\mathcal{D}} R^{\len(\varphi)}\cdot h(g,\varphi).
$$
Then:
\begin{prop}\label{prop:E-varrho-finite}
For all $R\in (0,\mathcal{R})$,
$$
\mathbb{E}_\varrho\bigl[R^{\mathrm{length}(\psi_1)}\bigr]<\infty.
$$
\end{prop}
\begin{proof}
Let be $R\in (0,\mathcal{R})$. Due to positive recurrence of $\bigl((\mathbf{W}_k,\psi_k)\bigr)_{k\in\N}$, we have
$$
\lim_{n\to\infty}\frac1n \sum_{j=1}^n R^{\mathrm{length}(\psi_j)} = \mathbb{E}_\varrho\bigl[R^{\mathrm{length}(\psi_1)}\bigr] \in [0,\infty] \quad \textrm{almost surely.}
$$
Define for $k,m\in\N$ and $\pi\in\Pi$
$$
h(\pi):=R^{\textrm{length}(\pi)},\quad
h_m(\pi) := R^{\textrm{length}(\pi)}\land m.
$$
We claim: for all $m\in\N$,
\begin{equation}\label{equ:E-hm-convergence}
\lim_{k\to\infty} \E\bigl[h_m(\psi_k)\bigr] = \sum_{(g,\varphi)\in\mathcal{D}} \varrho(g,\varphi)\cdot h_m(\varphi) =: \mathbb{E}_\varrho\bigl[h_m(\psi_1)\bigr]. 
\end{equation}
For the proof of the claim, we denote by $\mu_k$, $k\in\N$, the distribution of $(\mathbf{W}_k,\psi_k)$. Since $h_m$ is bounded, we have for all $(g,\varphi)\in\mathcal{D}$:
$$
h_m(\varphi)\cdot \mu_k(g,\varphi) \leq m\cdot \mu_k(g,\varphi).
$$
For $\varepsilon>0$, choose a finite subset $A_\varepsilon\subset \mathcal{D}$ such that
$$
\sum_{(g,\varphi) \in A_\varepsilon} \varrho(g,\varphi)\geq 1-\varepsilon.
$$
By the individual ergodic theorem, we have
\begin{equation}\label{equ:mu-k-convergence}
\mu_k(g,\varphi) \xrightarrow{k\to\infty} \varrho(g,\varphi)\quad\quad \textrm{ for every } (g,\varphi)\in\mathcal{D}.
\end{equation}
Moreover,
\begin{eqnarray*}
1&=&\sum_{(g,\varphi)\in A_\varepsilon} \varrho(g,\varphi)
+ \sum_{(g,\varphi)\in \mathcal{D}\setminus A_\varepsilon} \varrho(g,\varphi)\\
&=&
\sum_{(g,\varphi)\in A_\varepsilon} \mu_k(g,\varphi) + \sum_{(g,\varphi)\in \mathcal{D}\setminus A_\varepsilon} \mu_k(g,\varphi).
\end{eqnarray*}
Since $A_\varepsilon$ is finite, we can find, 
for any small $\delta>0$, some $k_\delta\in\N$ such that for all $k\geq k_\delta$ 
$$
\delta >
\biggl| \sum_{(g,\varphi)\in A_\varepsilon} \bigl(\varrho(g,\varphi) - \mu_k(g,\varphi)\bigr)\biggr| =
\biggl| \sum_{(g,\varphi)\in \mathcal{D}\setminus A_\varepsilon} \bigl(\varrho(g,\varphi) -  \mu_k(g,\varphi)\bigr)\biggr|
$$
Hence,
$$
\sum_{(g,\varphi)\in \mathcal{D}\setminus A_\varepsilon} \mu_k(g,\varphi)
\xrightarrow{k\to\infty} 
\sum_{(g,\varphi)\in \mathcal{D}\setminus A_\varepsilon} \varrho(g,\varphi)\leq \varepsilon.
$$
For $k$ sufficiently large, we have then
$$
\sum_{(g,\varphi)\in \mathcal{D}\setminus A_\varepsilon} \mu_k(g,\varphi) <2\varepsilon.
$$
Now:
\begin{eqnarray*}
&& \Bigl| \E\bigl[h_m(\psi_k)] - \E_\varrho\bigl[h_m(\psi_1)\bigr] \Bigr|\\
&=& \biggl| \sum_{(g,\varphi)\in\mathcal{D}} h_m(\varphi) \cdot \bigl( \mu_k(g,\varphi)-\varrho(g,\varphi)\bigr)\biggr| \\
&\leq & \biggl| \sum_{(g,\varphi)\in A_\varepsilon} h_m(\varphi) \cdot \bigl( \mu_k(g,\varphi)-\varrho(g,\varphi)\bigr)\biggr| \\
&&\quad + \biggl| \sum_{(g,\varphi)\in \mathcal{D}\setminus A_\varepsilon} h_m(\varphi) \cdot \bigl( \mu_k(g,\varphi)-\varrho(g,\varphi)\bigr)\biggr|.
\end{eqnarray*}
Since $A_\varepsilon$ is finite, we obtain together with (\ref{equ:mu-k-convergence}) that
$$
\sum_{(g,\varphi)\in A_\varepsilon} h_m(\varphi) \cdot  \mu_k(g,\varphi) \xrightarrow{k\to\infty} 
\sum_{(g,\varphi)\in A_\varepsilon} h_m(\varphi) \cdot \varrho(g,\varphi).
$$
On the other hand, for $k$ large enough,
$$
\sum_{(g,\varphi)\in \mathcal{D}\setminus A_\varepsilon} h_m(\varphi) \cdot  \mu_k(g,\varphi) \leq m\cdot \sum_{(g,\varphi)\in \mathcal{D}\setminus A_\varepsilon} \mu_k(g,\varphi) \leq 2\varepsilon m
$$
and
$$
\sum_{(g,\varphi)\in \mathcal{D}\setminus A_\varepsilon} h_m(\varphi) \cdot  \varrho(g,\varphi) \leq m\cdot \sum_{(g,\varphi)\in \mathcal{D}\setminus A_\varepsilon} \varrho(g,\varphi) \leq \varepsilon m.
$$
Hence, for $k$ large enough, we obtain
\begin{eqnarray*}
&& \Bigl|\E\bigl[h_m(\psi_k)\bigr] -\E_\varrho\bigl[h_m(\psi_1)\bigr] \Bigr|\\
&=&
\biggl|\sum_{(g,\varphi)\in\mathcal{D}} h_m(\varphi)\cdot \mu_k(g,\varphi) - \sum_{(g,\varphi)\in\mathcal{D}} h_m(\varphi)\cdot \varrho(g,\varphi)\biggr| \\
&\leq &\varepsilon + 3\varepsilon m \leq 4\varepsilon m.
\end{eqnarray*}
That is, we have proven the claim (\ref{equ:E-hm-convergence}) that $\E\bigl[h_m(\psi_k)\bigr]\to \E_\varrho\bigl[h_m(\psi_1)\bigr]$ as $k\to\infty$. \checkmark
\par
Recall from Proposition \ref{prop: E psi-k uniform finite} that there exists a non-negative real number $C_0<\infty$ such that $\sup_{k\in\N}\E\bigl[h(\psi_k)\bigr]\leq C_0$.
By definition and positivity of $h_m$, we have then for every $k\in\N$:
$$
\E\bigl[h_m(\psi_k)\bigr] \leq \E\bigl[h(\psi_k)\bigr] \leq C_0 <\infty.
$$
This together with $(\ref{equ:E-hm-convergence})$ implies that
$$
\E_\varrho\bigl[h_m(\psi_1)\bigr]\leq C_0<\infty.
$$
The Monotone Convergence Theorem finally yields
$$
\E_\varrho\bigl[R^{\textrm{length}(\psi_1)}\bigr] =
\E_\varrho\bigl[h(\psi_1)\bigr] =
\lim_{m\to\infty}\E_\varrho\bigl[h_m(\psi_1)\bigr] \leq C_0<\infty.
$$
\end{proof}
From the last proposition, we obtain the following consequence:
\begin{cor}\label{cor:expectation-psi2-restricted-finite}
Let $f:\Pi\to\mathbb{R}$ be a cone-additive function.    
    Suppose that there exist non-negative numbers $R < \mathcal{R}$ and $C<\infty$ with
    $$
    \bigl|f(\pi)\bigr|\leq C\cdot R^{\mathrm{length}(\pi)} \quad \textrm{for all $\pi\in\Pi$}.
    $$
    Then: 
    \begin{eqnarray*}
    \E_\varrho\bigl[\bigl|f(\psi_1|_{\neg C(\mathbf{W}_1)})\bigr|\bigr]&:=&\int|f(\psi_1|_{\neg C(\mathbf{W}_1)})\bigr| \,d\varrho \\ &=&\sum_{(g,\varphi)\in\mathcal{D}} \varrho(g,\varphi)\cdot |f(\varphi|_{\neg C(g})\bigr|
    <\infty.
    \end{eqnarray*}
\end{cor}
\begin{proof}
This follows immediately from Proposition \ref{prop:E-varrho-finite}:
$$
\E_\varrho\Bigl[\bigl|f(\psi_1|_{\neg C(\mathbf{W}_1)})\bigr|\Bigr]
\leq C\cdot \mathbb{E}_\varrho\Bigl[R^{\mathrm{length}(\psi_1|_{\neg C(\mathbf{W}_1)})}\Bigr]\leq C\cdot \mathbb{E}_\varrho\bigl[R^{\mathrm{length}(\psi_1)}\bigr] <\infty.
$$
\end{proof}

\begin{rem}
We note that the condition on the upper bound of $|f(\pi)|$ in Corollary \ref{cor:expectation-psi2-restricted-finite} is satisfied if $f(\pi)$ increases at most polynomially in the length of $\pi$ or, more generally, has sub-exponential growth w.r.t. $\mathrm{length}(\pi)$.
\end{rem}

From \cite{gilch:07} it is known that the process $\bigl((\mathbf{W}_k,\e_k-\e_{k-1})\bigr)_{k\in\N}$ is a homogeneous, irreducible, positive-recurrent Markov chain on some state space $\mathcal{D}_0\subseteq \bigl(V_1^\times\cup V_2^\times) \times \N$ with invariant distribution $\sigma$. Define
$$
\E_\sigma\bigl[\e_2-\e_1\bigr] := \int \e_2-\e_1 \, d\sigma = \sum_{(g,n)\in\mathcal{D}_0} n\cdot \sigma(g,n)<\infty,
$$
which is finite due to  \cite[Proof of Proposition 3.2]{gilch:07}. 
Then:
\begin{lemma}\label{lem: to E[T_2-T_1]}
$$
\lim_{n\to\infty} \frac{n}{\mathbf{k}(n)}=\mathbb{E}_\sigma[\e_2-\e_1] \quad \textrm{almost surely.}
$$
\end{lemma}
\begin{proof}
Since $\e_0=0$ we have
$$
\frac{e_{\mathbf{k}(n)}}{\mathbf{k}(n)}= \frac1{\mathbf{k}(n)}\sum_{i=0}^{\mathbf{k}(n)-1} (\e_{i+1}-\e_i).
$$
By \cite[Proposition 3.2]{gilch:07}, we have
$$
\lim_{n\to\infty} \frac{e_{\mathbf{k}(n)}}{\mathbf{k}(n)}=\mathbb{E}_\sigma[\e_2-\e_1]\quad \textrm{almost surely.}
$$
This implies
\begin{eqnarray*}
0&\leq & \frac{n-\e_{\mathbf{k}(n)}}{\mathbf{k}(n)}
\leq  \frac{\e_{\mathbf{k}(n)+1}-\e_{\mathbf{k}(n)}}{\mathbf{k}(n)} \\
&=&
\frac{\e_{\mathbf{k}(n)+1}}{\mathbf{k}(n)+1}\underbrace{\frac{\mathbf{k}(n)+1}{\mathbf{k}(n)}}_{\to 1 \textrm{ a.s.}}-\frac{\e_{\mathbf{k}(n)}}{\mathbf{k}(n)}
\xrightarrow{n\to\infty} 0\quad \textrm{almost surely.}
\end{eqnarray*}
Therefore,
$$
 \frac{n}{\mathbf{k}(n)} = \frac{n-\e_{\mathbf{k}(n)}}{\mathbf{k}(n)}+ \frac{\e_{\mathbf{k}(n)}}{\mathbf{k}(n)} \xrightarrow{n\to\infty} \mathbb{E}_\sigma[\e_2-\e_1]\quad \textrm{almost surely.}
$$
\end{proof}
The following lemma will be useful later for bounding  $f(\pi^\ast_{n})$ from above. 
\begin{lemma}\label{lem:R-psi-k-zero}
For $R\in(0,\mathcal{R})$, we have
$$
\lim_{k\to\infty} \frac{R^{\mathrm{length}(\psi_k)}}{k}=0 \quad \textrm{almost surely.}
$$
\end{lemma}
\begin{proof}
Let be $R\in(0,\mathcal{R})$. Then we obtain with the Ergodic Theorem for positive recurrent Markov chains the almost sure convergence
$$
\frac{R^{\mathrm{length}(\psi_k)}}{k} = \underbrace{\frac{1}{k}\sum_{j=1}^k R^{\mathrm{length}(\psi_j)}}_{\to \E_\varrho\bigl[R^{\mathrm{length}(\psi_1)}\bigr]<\infty}- \underbrace{\frac{k-1}{k}}_{\to 1}\underbrace{\frac{1}{k-1}\sum_{j=1}^{k-1} R^{\mathrm{length}(\psi_j)}}_{\to \E_\varrho\bigl[R^{\mathrm{length}(\psi_1)}\bigr]<\infty}
\xrightarrow{k\to\infty} 0.% \quad \textrm{a.s.}.
$$
\end{proof}
With the following  lemma, we want to control the term $f(\pi^\ast_{n})$ as $n\to\infty$.
\begin{lemma}\label{lem: end to 0}
Assume that $f:\Pi\to\mathbb{R}$ is a cone-additive function such that there exist  constants $C\in(0,\infty)$ and $R\in(0,\mathcal{R})$ with 
\begin{equation*} %\label{cond: rest bounded 2}
 \bigl|f(\pi)\bigr| \leq C\cdot R^{\mathrm{length}(\pi)} \quad \textrm{ for all } \pi\in\Pi.
\end{equation*}
Then:
$$
\lim_{n\to\infty} \frac{f\bigl(\pi^\ast_n\bigr)}{n}=0 \quad \textrm{almost surely.}
$$    
\end{lemma}
\begin{proof} 
First, observe that $\pi^\ast_n$ consists of the part of $[X_0,\ldots,X_n]$ which lies inside $C(X_{\e_{\mathbf{k}(n)}})$; this subpath will be inherited by $\psi_{\mathbf{e}_{\mathbf{k}(n)+1}}$ plus maybe some part which may be visited between times $n$ and $\e_{\mathbf{k}(n)+1}$, which becomes also part of $\psi_{\mathbf{e}_{\mathbf{k}(n)+1}}$. Therefore, 
we have for all $n\in\N$ that
$$
\mathrm{length}(\pi^\ast_n)\leq \mathrm{length}\bigl(\psi_{\kk(n)+1}\bigr),
$$
which in turn yields together with Lemmas 
\ref{lem: to E[T_2-T_1]} and \ref{lem:R-psi-k-zero} almost surely:
$$
\frac{\bigl|f\bigl(\pi^\ast_n\bigr)\bigr|}{n}
\leq C\cdot \frac{R^{\mathrm{length}(\pi_n^\ast)}}{n} 
\leq
C\cdot \underbrace{\frac{R^{\mathrm{length}(\psi_{\mathbf{k}(n)+1})}}{\mathbf{k}(n)+1}}_{\to 0}
\underbrace{\frac{\mathbf{k}(n)+1}{n}}_{\to \E_\sigma[\e_2-\e_1]<\infty}
\xrightarrow{n\to\infty}0.% \quad \textrm{a.s.}.
$$
\end{proof}

Once again, we recall that, by shift invariance of cone-additive functions, we have $f\bigl(\psi_2\bigl|_{\neg C(\mathbf{W}_2)}\bigr)=f(\pi_2)$.
Moreover, \cite[Theorem 3.3]{gilch:07} shows that the rate of escape w.r.t. the word length exists and is given by
$$
\ell=\lim_{n\to\infty}\frac{\Vert X_n\Vert}{n}=\frac1{\E_\sigma[\e_2-\e_1]}\in (0,\infty) \quad \textrm{ almost surely.}
$$ 
We now obtain:
\begin{lemma}\label{lem: to E[f(pi)]}
Assume that $\mathbb{E}_\varrho\Bigl[\bigl|f\bigl(\psi_1\bigl|_{\neg C(\mathbf{W}_1)}\bigr)\bigr|\Bigr]<\infty$. 
Then almost surely
$$
\lim_{n\to\infty}\frac{f\bigl([X_0,\ldots,X_n]\bigr)}{n}=
\frac{\mathbb{E}_\varrho\bigl[f\bigl(\psi_1\bigl|_{\neg C(\mathbf{W}_1)}\bigr)\bigr]}{\E_\sigma[\e_2-\e_1]}=\ell\cdot \mathbb{E}_\varrho\Bigl[f\bigl(\psi_1\bigl|_{\neg C(\mathbf{W}_1)}\bigr)\Bigr].% \quad \textrm{a.s..}
$$
\end{lemma}
\begin{proof}
For all $n\in\N$, according to (\ref{equ:f-decomposition}) we have
\begin{eqnarray*}
&& \frac1{n} f\bigl([X_0,\ldots,X_{n}]\bigr)\\
&=& \frac1{n}\cdot \biggl( 
f(\pi_1) + \sum_{i=2}^{\kk(n)} f(\pi_i) + f\bigl(\pi^\ast_{n}\bigr)\biggr) \\
&=& \frac1{n} f\bigl(\pi_1\bigr)+
\frac{\kk(n)}{n}\frac1{\kk(n)}
\sum_{i=2}^{\mathbf{k(n)}}f\bigl(\psi_i\bigl|_{\neg C(\mathbf{W}_i)}\bigr)+
\frac1{n}f\bigl(\pi^\ast_{n}\bigr).
\end{eqnarray*}
By Lemma \ref{lem: beginning to 0}, we have that $\frac1n f(\pi_1)\to 0$ almost surely as $n\to\infty$.
On the other hand side, by Lemma \ref{lem: end to 0}, we have
$\frac1n f(\pi^\ast_n)\to 0$ almost surely as $n\to\infty$.
The Ergodic Theorem for positive-recurrent Markov chains yields
$$
\lim_{n\to\infty} \frac1{\mathbf{k}(n)}
\sum_{i=2}^{\mathbf{k}(n)}f\Bigl(\psi_i\bigl|_{\neg C(\mathbf{W}_i)}\Bigr) = \mathbb{E}_\varrho\Bigl[f\bigl(\psi_1\bigl|_{\neg C(\mathbf{W}_1)}\bigr)\Bigr] \quad \textrm{almost surely.}
$$
Together with Lemma \ref{lem: to E[T_2-T_1]} we obtain
\begin{eqnarray}
\lim_{n\to\infty}\frac1n f\bigl([X_0,\ldots,X_n]\bigr) 
&=& \lim_{n\to\infty}\frac{\mathbf{k}(n)}{n}\frac1{\mathbf{k}(n)} \sum_{i=2}^{\kk(n)}f\Bigl(\psi_i\bigl|_{C(\mathbf{W}_i)}\Bigr)\nonumber
\\
&=& \frac{\mathbb{E}_\varrho\Bigl[f\Bigl(\psi_1\bigl|_{\neg C(\mathbf{W}_1)}\Bigr)\Bigr]}{\E_\sigma[\e_2-\e_1]} \label{equ:c-formula}\\
&=& \ell\cdot \mathbb{E}_\varrho\Bigl[f\Bigl(\psi_1\bigl|_{\neg C(\mathbf{W}_1)}\Bigr)\Bigr] \quad \textrm{almost surely.}\nonumber
\end{eqnarray}
\end{proof}

Finally, we can give the proof of our main result:
\begin{proof}[Proof of Theorem \ref{thm:limit-theorem}]
By assumption there exist constants $C\in(0,\infty)$ and $R\in(0,\mathcal{R})$ such that $\bigl|f(\pi)\bigr|\leq R^{\mathrm{length}(\pi)}$  for all $\pi\in\Pi$. Corollary \ref{cor:expectation-psi2-restricted-finite} provides that
$$
\mathbb{E}_\varrho\Bigl[\bigl|f\bigl(\psi_1\bigl|_{\neg C(\mathbf{W}_1)}\bigr)\bigr|\Bigr]<\infty.
$$
Lemma \ref{lem: to E[f(pi)]} yields now existence of some $\mathfrak{c}\in\mathbb{R}$ such that
$$
\lim_{n\to\infty} \frac{f\bigl([X_0,\ldots,X_n]\bigr)}{n}= \mathfrak{c}:=\frac{\mathbb{E}_\varrho\Bigl[f\Bigl(\psi_1\bigl|_{\neg C(\mathbf{W}_1)}\Bigr)\Bigr]}{\E_\sigma[\e_2-\e_1]}\quad \textrm{ almost surely.}
$$
If $f$ is strictly positive, then it follows immediately from the formula for the limit $\mathfrak{c}$ that $\mathfrak{c}>0$, since $\e_2-\e_1>0$ almost surely and $\ell>0$ according to \cite[Theorem 3.3]{gilch:07}.
\end{proof}
In particular, we have shown:
\begin{cor}\label{cor:c-formula}
$$
\mathfrak{c}= \frac{\mathbb{E}_\varrho\Bigl[f\Bigl(\psi_1\bigl|_{\neg C(\mathbf{W}_1)}\Bigr)\Bigr]}{\E_\sigma[\e_2-\e_1]} \\
= \ell\cdot \mathbb{E}_\varrho\Bigl[f\Bigl(\psi_1\bigl|_{\neg C(\mathbf{W}_1)}\Bigr)\Bigr].
$$
\end{cor}
\begin{proof}
This follows immediately from (\ref{equ:c-formula}) in the proof of Lemma \ref{lem: to E[f(pi)]}.
\end{proof}

\section{Applications}
\label{sec:applications}

In this section, we present applications of Theorem \ref{thm:limit-theorem}.
Important asymptotic random walk quantities such as entropy, drift, range, and certain travelling salesman problems on free products of graphs share a common property: they can be viewed as trajectory-wise limits of the function values of cone-additive functions. 
In this section, we will see old and completely new results, all derived from the viewpoint of cone-additive functions.
\par
In all of the below, we assume that the random walk $(X_n)_{n\in\mathbb{N}_0}$ on $V$ is such that the convergence radius $\mathscr{R} $ of the Green function $ G(\id, \id\vert z) $ is strictly bigger than $1$. Recall also once again the definitions of $\pi_k$, $k\in\N$, and $\pi_n^\ast$, $n\in\N$, from Subsection \ref{subsec: paths}

\subsection{Entropy} 
%Green metric.  Benjamini-Peres

Denote by $Q_n$ the distribution of $X_n$. The \textit{asymptotic entropy} of $(X_n)_{n\in\N_0}$ is given by 
$$
h=\lim_{n\to\infty}-\frac1n \log Q_n(X_n),
$$
which exists due to \cite{gilch:11}. In the following, we will show with the help of a suitable cone-additive function that this limit exists. 
\par
For $v,w\in V$ and $z\in\mathbb{C}$, the \textit{last visit generating function} is defined as 
$$
L(v,w|z):=\sum_{n\geq 0} \P_v\bigl[X_n=w,\forall m\in\{1,\ldots,n\}: X_m\neq v \bigr]\cdot z^n.
$$ 
For any path $\pi=[u_0,\ldots,u_n]\in\Pi$, we define 
$$
f:\Pi\to\mathbb{R}, \pi\mapsto
f(\pi):=-\log\bigl( L(u_0,u_n|1)\bigr).
$$ 

\begin{lemma} The function $f$ is cone-additive. % (Definition~\ref{def: cone add}).
\end{lemma}
\begin{proof}
First, we make the following crucial observation:
if every path from $u\in V$ to $w\in V$ has to pass through $v\in V$, then 
\begin{equation}\label{equ:L-decomposition}
L(u,w|1)=L(u,v|1)\cdot L(v,w|1);
\end{equation}
this can be easily checked by conditioning on the last visit of $v$ when walking from $x$ to $w$. In our setting every path
 from $X_0$ to $X_n$ has to pass through $X_{\e_1}, X_{\e_2},\ldots, X_{\e_{\mathbf{k}(n)}}$. Furthermore, observe that $\pi_i$, $i\in\N$, is a path from $X_{\e_{i-1}}$ to $X_{\e_i}$, where $\e_0=0$, and that $\pi^\ast_n$ is a path from $X_{\e_{\kk(n)}}$ to $X_n$.
Therefore, an iterated application of (\ref{equ:L-decomposition}) gives
 \begin{eqnarray*}
   &&  f\bigl([X_0, \ldots,X_n]\bigr)=  -\log L(X_0,X_n|1) \\
&=& -\log\left ( \prod_{i=1}^{\mathbf{k}(n)} L(X_{\e_{i-1}},X_{\e_{i}}|1)\cdot L(X_{\e_{\mathbf{k}(n)}},X_n|1)\right)
 \\ 
 &= & -\sum_{i=1}^{\mathbf{k}(n)} \log L(X_{\e_{i-1}},X_{\e_{i}}|1) - \log L(X_{\e_{\mathbf{k}(n)}},X_n|1) \\
 &=&
 \sum_{i=1}^{\mathbf{k}(n)}f(\pi_i) + f(\pi_n^*).
\end{eqnarray*}
 This proves the additivity property in the definition of  cone-additive functions.
 \par
For the proof of shift invariance of $f$, we remark that, for $u,v\in V$ with $v\in C(u)$, we obtain with Lemma \ref{lem:cone-probs} that
\begin{eqnarray*}
L(u,v|1) &=&\sum_{n\geq 0} \P_{u}\bigl[X_n=v,\forall m\in\{1,\ldots,n\}: X_m\neq u \bigr] \\
&=& \sum_{n\geq 0} \P_{o}\bigl[X_n=u^{-1}v,\forall m\in\{1,\ldots,n\}: X_m\neq o \bigr] \\
&=& L(o,u^{-1}v|1),
\end{eqnarray*}
since the probabilities $\P_{u}\bigl[X_n=v,\forall m\in\{1,\ldots,n\}: X_m\neq u \bigr]$ take into account only paths within $C(u)$. Since $\pi_i$, $i\in\N$, is a path from $X_{\e_{j-1}}$ to $X_{\e_j}\in C(X_{\e_{j-1}})$, we get
\begin{eqnarray*}
f(\pi_j) &=& -\log L(X_{\e_{j-1}},X_{\e_j}|1) \\
&=& -\log L(o,X_{\e_{j-1}}^{-1}X_{\e_j}|1) = f\bigl(X_{\e_{j-1}}^{-1}\pi_j \bigr),
\end{eqnarray*}
that is, we have proven shift invariance. 
\end{proof}
 
In the following, we assume now that  there exists some $\varepsilon_0>0$ such that $p(v,w)>\varepsilon_0$ for all $v,w\in V$ with $p(v,w)\neq 0$. This yields that, for any $\pi=[u_0,\ldots,u_n]\in\Pi$,
$$
\bigl|f(\pi)\big| = \bigl| \log L(u_0,u_n)\bigr| \leq \bigl| \log \varepsilon_0^n\bigr| \leq n\cdot \bigl|\log\varepsilon_0\bigr|,
$$
that is, Condition (\ref{cond: rest bounded}) is satisfied and Theorem \ref{thm:limit-theorem} can be applied which ensures existence of a constant $\mathfrak{c}_{\mathrm{entropy}}>0$ such that
$$
\mathfrak{c}_{\mathrm{entropy}}=\lim_{n\to\infty}\frac1n f\bigl([X_0,\ldots,X_n]\bigr)= \lim_{n\to\infty} -\frac1n \log L(o,X_n) \quad \textrm{ almost surely.} 
$$
In \cite{gilch:11} it is shown that the above limit coincides with the asymptotic entropy, that is, 
$$
h=\lim_{n\to\infty}-\frac1n \log Q_n(X_n)=\lim_{n\to\infty}-\frac1n \log L(o,X_n).
$$

\subsection{Range}

\subsubsection{The Number of Visited Vertices}

The \textit{range} of the random walk $(X_n)_{n\in\N_0}$ on $V$ until time \mbox{$n\in\N_0$} is given by
$$
\mathbf{R}_n:= \bigl|\{X_{0},\ldots, X_n\}\bigr|,
$$
the number of distinct vertices which are visited up to time $n$. The \textit{asymptotic range} is given by the almost sure constant limit
$$
\mathfrak{r}=\lim_{n\to\infty} \frac{\mathbf{R}_n}{n},
$$
whose existence was shown in \cite{gilch:22}. We now demonstrate the existence of this limit with the help of the following function
$$
R: \Pi\to\mathbb{R},  [u_0,\ldots,u_n] \mapsto \bigl| \{u_0,\ldots, u_n\}\setminus\{u_0\}\bigr|=\bigl| \{u_0,\ldots, u_n\}\bigr|-1.
$$
\begin{lemma}
  The function  $R$ is cone-additive.
\end{lemma}

\begin{proof}
Let be $n\in\N_0$ and consider the random walk path $[X_0,\ldots,X_{n}]$ up to time $n$. 
Let $V_1,\ldots V_{\mathbf{R}_n}$ denote the distinct (random) vertices in $\{X_0, X_1,\ldots, X_n\}$. Then $R\bigl([X_0,\ldots,X_n]\bigr)=\mathbf{R}_n-1$.
\par
Now we make the observation that there are three possibilities for each $V_j$, $j\in\{1,\ldots,\mathbf{R}_n\}$: either $V_j$ is a cone entry point, or it lies properly inside a sphere $\mathcal{S}_i$, $i\in\{1,\ldots,\kk(n)\}$, or $V_j\neq X_{\e_{\kk(n)}}$ belongs to $\pi_n^\ast$. More precisely, either 
$V_j= X_{\e_i}$ for some $i\in \{0,\ldots, \mathbf{k}(n)\}$ or 
$$
V_j\in \mathcal{S}_i^o:=C(X_{\e_{i-1}})\setminus \left( C(X_{\e_{i}}) \cup \{X_{\e_{i-1}}\} \right)
$$ 
for some $i\in \{1,\ldots, \mathbf{k}(n)\}$, or $V_j\in C(X_{\e_{\kk(n)}})\setminus\{X_{\e_{\kk(n)}}\}$.
If $V_j\in \mathcal{S}_i^o$ lies properly inside a sphere $\mathcal{S}_i$, then it appears in exactly one of the pieces $\pi_{1}, \ldots,\pi_{\mathbf{k}(n)}$. Thus, $V_j$ is counted in exactly one out of $R(\pi_1),\ldots,R(\pi_{\kk(n)})$.
If $V_j\in C(X_{\e_{\kk(n)}})\setminus\{X_{\e_{\kk(n)}}\}$ then $V_j$ appears only in $\pi_n^\ast$ and it is counted in 
 $R\bigl([X_{\e_{\kk(n)}},\ldots,X_n]\bigr)$.
\par
If $V_j =X_{\e_i}$ for some $i\in \{1,\ldots, \mathbf{k}(n)-1\}$, then 
$V_j$ appears in $\pi_{i}$ and $\pi_{i+1}$; if $V_j=X_{\e_0}=o$, then $V_j$  only appears  in $\pi_1$; if $V_j=X_{\e_{\kk(n)}}$ then $V_j$ appears in $\pi_{\kk(n)}$ and $\pi_n^*$. 
The subtraction by 1 in the definition of the function $R(\cdot)$ takes care of this double-counting, which we will show now. 
\par
Let $\{V_1^{(i)},\ldots, V_{m_i}^{(i)}\}$ denote those vertices in $\{V_1,\ldots, V_{\mathbf{R}_n}\}$ which lie inside $\mathcal{S}_i$ for $i\in\{1,\ldots, \mathbf{k}(n)\}$,
  and let $V_{m_i}^{(i)}=X_{\e_i}=V_{1}^{(i+1)}$ be the cone entry vertices (which -- by construction -- lie in $\pi_i$ and $\pi_{i+1}$) for $i<\mathbf{k}(n)$ and $V_{m_{\mathbf{k}(n)}}^{(\kk(n))}=X_{\e_{{\mathbf{k}(n)}}}$.
  Let $\{V_1^\ast,\ldots, V_{m^\ast}^\ast\}$  denote those vertices which lie in $\pi_n^\ast$ with $V_1^\ast=X_{\e_{\mathbf{k}(n)}}$.
 We now observe that 
 $$
 \sum_{i=1}^{\mathbf{k}(n)} \bigl|\{V_1,\ldots, V_{\mathbf{R}_n}\}\cap \mathcal{S}_i\bigr|+ \bigl|\{V_1^\ast,\ldots, V_{m^\ast}^\ast\}\bigr|=
 \sum_{i=1}^{\mathbf{k}(n)} m_i +m^\ast= \mathbf{R}_n+\mathbf{k}(n), 
 $$ since $\mathbf{k}(n)$ elements (namely, $X_{\e_1},\ldots,X_{\e_{\kk(n)}}$) are counted twice.
Thus, we get 
\begin{eqnarray*}
    \sum_{i=1}^{\mathbf{k}(n)} R(\pi_i)+R(\pi_n^\ast)&=& \sum_{i=1}^{\mathbf{k}(n)} (m_i -1)+m^\ast-1\\
&=& \mathbf{R}_n+\mathbf{k}(n) -(\mathbf{k}(n) +1)\\
&=& \mathbf{R}_n-1=R\bigl([X_0,\ldots,X_n]\bigr),
\end{eqnarray*}
which shows the additivity property.
\par
The shift invariance follows since the number of distinct points visited in $\pi_j=[U_0^{(j)},\ldots,U_{t_j}^{(j)}]$, $j\in\N$, does not change when we cancel the prefix $X_{\e_{j-1}}$ in each vertex of $\pi_j$, that is, 
\begin{eqnarray*}
R(\pi_j)&=&\bigl|\bigl\{U_0^{(j)},\ldots,U_{t_j}^{(j)}\bigr\}\bigr|-1 \\
&=& \bigl|\bigl\{X_{\e_{j-1}}^{-1}U_0^{(j)},\ldots,X_{\e_{j-1}}^{-1}U_{t_j}^{(j)}\bigr\}\bigr|-1 = R\bigl( X_{\e_{j-1}}^{-1}\pi_j\bigr).
\end{eqnarray*}
\end{proof}

Moreover, we have $R\bigl([X_0,\ldots,X_n]\bigr)\leq n+1$, thus Condition (\ref{cond: rest bounded}) does hold and  Theorem~\ref{thm:limit-theorem} yields existence of $\mathfrak{c}_{\mathrm{range}}>0$ with
$$ 
\mathfrak{c}_{\mathrm{range}}=\lim_{n\to\infty} \frac{R\bigl([X_0,\ldots,X_n]\bigr)}{n}= \lim_{n\to\infty} \frac{\mathbf{R}_n}{n}=\mathfrak{r} \quad \textrm{almost surely.}
$$

\subsubsection{Edge-Range of $r$ Visits}
Let us now focus on edges of the graph $\mathcal{X}$ instead of vertices in $V$. We want to count the number of edges which are visited exactly $r\in\N$ times until time $n$. Hereby, we  identify (oriented) edges by paths of length $2$. For this purpose, we denote by $\mathbf{E}_{n}^{(r)}$ the number of distinct edges $[X_{i-1},X_{i}]$, $i\in\{1,\ldots,n\}$, which are traversed  exactly $r$ times up to time $n$. 
More precisely, for $[u,v]\in \Pi$ with $p(u,v)>0$ and $[u_0,\ldots,u_m]\in \Pi$, we set
$$
N([u,v],\pi):=\Bigl| \bigl\{ j\in\{1,\ldots,m\} \,\bigl|\, [u_{j-1},u_j]=[u,v]\bigr\}\Bigr|,
$$
the number of visits of the edge $[u,v]$ in $\pi$.
Then
$$
\mathbf{E}_{n}^{(r)} =\sum_{e \in \bigl\{[X_{j-1},X_j] \,\bigl|\, j\in\{1,\ldots,n\}\bigr\}} \mathds{1}_{\{r\}}\bigl(N(e,[X_0,\ldots,X_n])\bigr).
$$

We are interested in the asymptotic behaviour of $\frac1n \mathbf{E}^{(r)}_{n}$ as $n\to\infty$.
For this purpose, we define the function
$$
R_{\text{edges}}^{(r)}:\Pi \to \mathbb{R}, \pi=[u_0,\ldots,u_m]\mapsto \sum_{e \in \bigl\{[u_{j-1},u_j] \,\bigl|\, j\in\{1,\ldots,m\}\bigr\}} \mathds{1}_{\{r\}}\bigl(N(e,\pi)\bigr),
$$ 
the number of distinct edges $[u_{i-1},u_i]$, $i\in\{1,\ldots,m\}$, of $\pi$ which are traversed exactly $r$ times by $\pi$.
\begin{lemma}\label{lem: edge-range cone add}
  The function  $R_{\text{edges}}^{(r)}$ is cone-additive. 
\end{lemma}

\begin{proof}
The splitting of $[X_0,\ldots,X_n]$, $n\in\N$, into $\pi_1,\ldots,\pi_{\mathbf{k}(n)}, \pi_n^\ast $
%(\pi_i)_{i=1}^{\mathbf{k}(n)}
results in a disjoint partition w.r.t. the edges, that is, every traversed edge $[X_{i-1},X_i]$ lies in exactly one of the spheres $\mathcal{S}_j$ defined in Subsection~\ref{subsec: paths}. Thus, every visit of this edge happens only inside $\pi_j$. Therefore,
$$
R_{\text{edges}}^{(r)}\bigl([X_0,\ldots,X_n]\bigr)=\sum_{j=1}^{\mathbf{k}(n)} R_{\text{edges}}^{(r)}(\pi_j)+ R_{\text{edges}}^{(r)}(\pi_n^\ast) .
$$
This proves the additivity property.
The shift invariance follows from the construction of the free product, that is, $[X_{i-1},X_i]$ is a visited edge in $\pi_j$ if and only if $\bigl[X_{\e_{j-1}}^{-1}X_{i-1},X_{\e_{j-1}}^{-1}X_i\bigr]$ is a visited edge in $X_{\e_{j-1}}^{-1}\pi_j$.
\end{proof}

Now we obtain the following new result:
\begin{thm}\label{thm: range}  
Assume that $\mathscr{R}>1$.
Then there exists a constant $\mathfrak{c}_{edges}>0$ such that
$$
\mathfrak{c}_{edges}=\lim_{n\to\infty} \frac{\mathbf{E}^{(r)}_{n}}{n}
=\lim_{n\to\infty} \frac{R_{\text{edges}}^{(r)}\bigl([X_0,\ldots,X_n]\bigr)}{n}
\quad \textrm{ almost surely.}
$$
\end{thm}

\begin{proof}
By Lemma~\ref{lem: edge-range cone add}, $R^{(r)}_{\text{edges}}$ is cone-additive. Moreover, we obviously have $R^{(r)}_{\text{edges}}\bigl([X_0,\ldots,X_n]\bigr)\leq n$, thus  Condition (\ref{cond: rest bounded}) in Theorem~\ref{thm:limit-theorem} does hold, and the claim follows now immediately from that theorem.
\end{proof}

\begin{remark}[Generalisations]
The above theorem remains true when replacing $R_r^{\text{edges}}$ by other quantities, like the number of edges visited at least $r$-times, at most $r$-times, an even/odd number of times, and so on.
\end{remark}

\subsubsection{Range of $r$ Visits}
In the case when we want to count the number of vertices which are visited exactly $r\in\N$ times by the random walk $(X_n)_{n\in\N_0}$, the situation becomes a little bit more complicated. We are not able to apply Theorem \ref{thm:limit-theorem} directly; however, we can use several results from Section \ref{sec:proof-of-LLT} to prove an analogous limit theorem.
\par 
For this purpose, define for $u\in V$ and $[u_0,\ldots,u_m]\in \Pi$, $m\in\N$,
$$
N(u,\pi):=\bigl| \bigl\{ j\in\{0,1,\ldots,m\} \,\bigl|\, u_j=u\bigr\}\bigr|,
$$
the number of visits of $u$ in $\pi$.
Then
$$
\mathbf{R}_{n}^{(r)} =\sum_{u \in \{X_0,\ldots,X_n\}} \mathds{1}_{\{r\}}\bigl(N(u,[X_0,\ldots,X_n])\bigr)
$$
is the number of vertices in $[X_0,\ldots,X_n]$ which are visited exactly $r$ times up to time $n$.
We are interested in the asymptotic behaviour of $\frac1n \mathbf{R}^{(r)}_{n}$ as $n\to\infty$.
For this purpose, we define the auxiliary function
$$
R^{(r)}:\Pi \to \mathbb{R},  \pi=[u_0,\ldots,u_m]\mapsto \sum_{u \in \{u_0,\ldots,u_{m}\}} \mathds{1}_{\{r\}}\bigl(N(u,\pi)\bigr),
$$ 
the number of distinct vertices in $\{u_0,\ldots,u_{m}\}$ which are visited exactly $r$ times by $\pi$. 
%The starting point $u_0$ and the endpoint $u_m$ are not taken into account since they will play a special role as we will see later.
\par
Recall the definition of $\mathcal{D}$ from Subsection \ref{subsec:process-Wk-psik}. We define the function $R_2^{(r)}: \mathcal{D}^2 \to\N_0$ by
\begin{eqnarray*}
&&R_2^{(r)} \bigl((g_1,\varphi_1),(g_2,\varphi_2)\bigr)\\
&:=&
\sum_{v\in \{v_1,\ldots,v_{m-1}\}\setminus\{v_0,v_m\}} \mathds{1}_{\{r\}}\bigl(N(v,\varphi_1\bigl|_{\neg C(g_1)})\bigr)\\
&&\quad +\mathds{1}_{\{r\}}\Bigl(N(g_1,\varphi_1\bigl|_{\neg C(g_1)})+N(o,\varphi_2\bigl|_{\neg C(g_2)})-1 \Bigr),
\end{eqnarray*}
where $\varphi_1\bigl|_{\neg C(g_1)}=[v_0,\ldots,v_m]$, $m\in\N$.

In the present setting, we will show a property of $R^{(r)}$ which is an adapted version of cone-additivity. For this purpose, define for $n\in\N$
\begin{eqnarray*}
&&\widehat R^{(r)}\bigl([X_0,\ldots,X_n]\bigr) \\
&:=& R^{(r)}(\pi_1) + \sum_{i=2}^{\kk(n)} R_2^{(r)}\bigl((\W_i,X_{\e_{i-1}}^{-1}\pi_i),(\W_{i+1},X_{\e_{i}}^{-1}\pi_{i+1})\bigr)+ R^{(r)}(\pi_n^\ast).
\end{eqnarray*}

\begin{lemma}
For all $n\in\N$, we have almost surely
$$
\Bigl| R^{(r)}\bigl([X_0,\ldots,X_n]\bigr)-\widehat R^{(r)}\bigl([X_0,\ldots,X_n]\bigr)\Bigr|\leq 3.
$$
\end{lemma}
\begin{proof}
Let be $n\in\N$.
For any $X_j$, $j\in\{0,\ldots,n\}$, we have that either 
$$
X_j\in \mathcal{S}_i^o:=C(X_{\e_{i-1}})\setminus \bigl( C(X_{\e_i})\cup\{X_{\e_{i-1}}\}\bigr) \quad \textrm{for some $i\in\{1,\ldots,\kk(n)\}$},
$$  
or $X_j=X_{\e_i}$ for some $i\in\{0,\ldots,\kk(n)\}$ or $X_j\in C(X_{\e_{\kk(n)}})\setminus \{X_{\e_{\kk(n)}}\}$.  Observe that $\mathcal{S}_1^o,\ldots,\mathcal{S}_{\kk(n)}^o$ are pairwise disjoint.  
\par
In the first case $X_j\in\mathcal{S}_i^o$  appears only in $\pi_i$, and $X_j$ is visited in $\pi_i$ exactly $r$ times if and only if $X_{\e_{i-1}}^{-1}X_{j}$ is visited exactly $r$ times in $X_{\e_{i-1}}^{-1}\pi_i$. That is, if $X_{j}\in\mathcal{S}_i^o$ is visited exactly $r$ times then it is exactly counted once in $R^{(r)}(\pi_1)$, if $i=1$, or in $R_2^{(r)}\bigl((\W_i,X_{\e_{i-1}}^{-1}\pi_i),(\W_{i+1},X_{\e_{i}}^{-1}\pi_{i+1})\bigr)$ for $i\in\{2,\ldots,\kk(n)\}$. 
\par
If $X_j\in C(X_{\e_{\kk(n)}})\setminus \{X_{\e_{\kk(n)}}\}$  is visited exactly $r$ times up to time $n$, then it is counted once in $R^{(r)}(\pi_n^\ast)$. 
\par
If $X_j=X_{\e_i}$ for some $i\in\{2,\ldots,\kk(n)-1\}$ and $X_j$ is exactly visited $r$ times then it is counted once by  $\mathds{1}_{\{r\}}\bigl(N(\W_i,X_{\e_{i-1}}^{-1}\pi_i)+N(o,X_{\e_{i}}^{-1}\pi_{i+1})-1 \bigr)$. In this case the number of visits of $X_{\e_i}$ splits up into the number of visits to $X_{\e_i}$, where the preceding step comes from the outside of $C(X_{\e_i})$ (counted in $\pi_i$) or from the interior of $C(X_{\e_i})$ (counted in $\pi_{i+1}$). Since the last point of $\pi_i$ corresponds to the first point of $\pi_{i+1}$, we have to subtract one count of $X_{\e_i}$.
%This explains why we excluded the starting and end points in the definition of $R_2^{(r)}$. 
\par
The only points not considered so far are $X_0$, $X_{\e_1}$, and $X_{\e_{\kk(n)}}$ which play some special role: each of this points may be counted at most once in $R^{(r)}\bigl([X_0,\ldots,X_n\bigr])$ but maybe not in $\widehat R^{(r)}\bigr([X_0,\ldots, X_n]\bigr)$, or vice versa.
% $X_{\e_1}$ may be counted once in $R^{(r)}\bigl([X_0,\ldots,X_n\bigr])$ while it is definitely not counted in $\widehat R^{(r)}\bigr([X_0,\ldots, X_n]\bigr)$; the point $X_{\e_{\kk(n)}}$, however, is maybe counted in  $R_2^{(r)}\bigl((\W_{\kk(n)},X_{\e_{\kk(n)-1}}^{-1}\pi_{\kk(n)}),(\W_{\kk(n)+1},X_{\e_{\kk(n)}}^{-1}\pi_{\kk(n)+1})\bigr)$ but maybe it is not counted in $R^{(r)}\bigl([X_0,\ldots,X_n\bigr])$. 
This explains the possible difference of at most $3$ between $R^{(r)}\bigl([X_0,\ldots,X_n\bigr])$ and $\widehat R^{(r)}\bigr([X_0,\ldots, X_n]\bigr)$, which finally proves the proposed inequality. 
\end{proof}
Since we obviously have $\mathbf{R}_n^{(r)} =R^{(r)}\bigl([X_0,\ldots,X_n]\bigr)$
% $$
% \mathbf{R}_n^{(r)} = R^{(r)}\bigl([X_0,\ldots,X_n]\bigr)+
% \sum_{u\in\{o,X_n\}}
% \mathds{1}_{\{r\}}\bigl(N(u,[X_0,\ldots,X_n])\bigr)
% $$
%+\mathds{1}_{\{r\}}\bigl(N(X_n,[X_0,\ldots,X_n])\bigr)\\
the last lemma yields
\begin{equation}\label{equ:range-estimate}
\Bigl| \mathbf{R}_n^{(r)} -\widehat R^{(r)}\bigl([X_0,\ldots,X_n]\bigr)\Bigr|\leq 3.
\end{equation}

Now we can prove the next new result:

\begin{thm} \label{thm:r-range}  
Assume that $\mathscr{R}>1$.
Then there exists a constant $\mathfrak{c}_{\mathrm{range}}^{(r)}>0$ such that
$$
\mathfrak{c}_{\mathrm{range}}^{(r)}=\lim_{n\to\infty} \frac{\mathbf{R}_{n}^{(r)}}{n}
=\lim_{n\to\infty} \frac{R^{(r)}\bigl([X_0,\ldots,X_n]\bigr)}{n}
\quad \textrm{ almost surely.}
$$
\end{thm}
\begin{proof}
First, observe that $\Bigl(\bigl((\W_k,\psi_k),(\W_{k+1},\psi_{k+1})\bigr)\Bigr)_{k\in\N_0}$ is also a positive-recurrent, irreducible Markov chain with equilibrium 
$$
\varrho_2\bigl((g_1,\varphi_1),(g_2,\varphi_2)\bigr):=
\varrho(g_1,\varphi_1)\cdot q\bigl((g_1,\varphi_1),(g_2,\varphi_2)\bigr),
$$
where $\varrho(\cdot)$ is the equilibrium of $\bigl((\W_k,\psi_k)\bigr)_{k\in\N_0}$ and $q(\cdot,\cdot)$ describes its transition probabilities. 
\par
Obviously, we have
$$
\lim_{n\to\infty} \frac{R^{(r)}(\pi_1)}{n}=0 \quad \textrm{ almost surely},
$$
since $R^{(r)}(\pi_1)\leq \e_1<\infty$ almost surely, and
$$
\lim_{n\to\infty} \frac{R^{(r)}(\pi_n^\ast)}{n}=0 \quad \textrm{ almost surely},
$$
which follows from $R^{(r)}(\pi_n^\ast)\leq \mathrm{length}(\pi_n^\ast)$ together with Lemma \ref{lem: end to 0}. Furthermore, the Ergodic Theorem for positive recurrent Markov chains yields
\begin{eqnarray*}
&&\frac1n \sum_{i=2}^{\kk(n)} R_2^{(r)}\bigl((\W_i,X_{\e_{i-1}}^{-1}\pi_i),(\W_{i+1},X_{\e_{i}}^{-1}\pi_{i+1})\bigr) \\
&=&\frac1n \sum_{i=2}^{\kk(n)} R_2^{(r)}\bigl((\W_i,\psi_i\bigl|_{\neg C(\W_i)}),(\W_{i+1},\psi_{i+1}\bigl|_{\neg C(\W_{i+1})})\bigr) \\
&\xrightarrow{n \to\infty} &
\underbrace{\sum_{\substack{(g_1,\varphi_1),\\(g_2,\varphi_2)\in\mathcal{D}}}\varrho_2\bigl((g_1,\varphi_1),(g_2,\varphi_2)\bigr) R_2^{(r)}\bigl((g_1,\varphi_1\bigl|_{\neg C(g_1)}),(g_2,\varphi_2\bigl|_{\neg C(g_2)})\bigr)}_{=:\mathfrak{c}_{\mathrm{range}}^{(r)}},
\end{eqnarray*}
where the sum on the right-hand side is finite since
\begin{eqnarray*}
&&\sum_{\substack{(g_1,\varphi_1),\\(g_2,\varphi_2)\in\mathcal{D}}}\varrho_2\bigl((g_1,\varphi_1),(g_2,\varphi_2)\bigr) \cdot R_2^{(r)}\bigl((g_1,\varphi_1\bigl|_{\neg C(g_1)}),(g_2,\varphi_2\bigl|_{\neg C(g_2)})\bigr)\\
&\leq & \sum_{(g_1,\varphi_1),(g_2,\varphi_2)\in\mathcal{D}}\varrho_2\bigl((g_1,\varphi_1),(g_2,\varphi_2)\bigr)\cdot \mathrm{length}(\varphi_1\bigl|_{\neg C(g_1)})\\
&=& \sum_{(g_1,\varphi_1)\in\mathcal{D}}\varrho\bigl((g_1,\varphi_1)\bigr) \cdot \mathrm{length}(\varphi_1\bigl|_{\neg C(g_1)})<\infty,
\end{eqnarray*}
where finiteness follows from Corollary \ref{cor:expectation-psi2-restricted-finite}. Therefore,
$$
\mathfrak{c}_{\mathrm{range}}^{(r)}=\lim_{n\to\infty} \frac{\widehat R^{(r)}\bigl([X_0,\ldots,X_n]\bigr)}{n}\quad \textrm{ almost surely.}
$$
From (\ref{equ:range-estimate}) follows now the proposed claim that
$$
\mathfrak{c}^{(r)}=\lim_{n\to\infty} \frac{\mathbf{R}^{(r)}_n}{n}\quad \textrm{ almost surely.}
$$
\end{proof}

\subsection{Drift}

We present some limit theorems in the context of the drift of the random walk $(X_n)_{n\in\N_0}$.

\subsubsection{Graph Distance}
For $u,v\in V$, we write
$$
d_\mathcal{X}(u,v):= \min\bigl\{ \mathrm{length}(\pi) \,\bigl|\, \pi \textrm{ is a path from $u$ to $v$}\bigl\}
$$
for the graph distance of $u$ and $v$ in $\mathcal{X}$. We set
$$
f:\Pi \to \N_0, \, [u_0,\ldots,u_m] \mapsto d_\mathcal{X}(u_0,u_m),
$$
the graph distance between the starting and end point of $\pi$.

\begin{lemma}
    The function $f$ is cone-additive.
\end{lemma}

\begin{proof}
Let be $n\in\N$ and consider $[X_0,\ldots,X_n]$. Then any  path from $X_0$ to $X_n$ has to pass through $X_{\e_1},\ldots,X_{\e_{\kk(n)}}$. Recall that $\pi_j$, $j\in\N$, is a path from $X_{\e_{j-1}}$ to $X_{\e_j}$ and $\pi_n^\ast$ is a path from $X_{\e_{\kk(n)}}$ to $X_n$.
Therefore, we have
\begin{eqnarray*}
f\bigl([X_0,\ldots,X_n]\bigr) &=& d_\mathcal{X}(X_0,X_n) =
\sum_{j=1}^{\kk(n)} d_\mathcal{X}(X_{\e_{j-1}},X_{\e_j}) + d_\mathcal{X}(X_{\e_{\kk(n)}},X_n)\\
&=& \sum_{j=1}^{\kk(n)} f(\pi_j ) + f(\pi_n^\ast),
\end{eqnarray*}
which proves cone additivity.
\par
For the proof of shift-invariance,  write $\pi_j=\bigl[U_0^{(j)},\ldots,U_{t_j}^{(j)}\bigr]$, $j\in\N$, and observe that all $U_i^{(j)}$, $i\in\{0,\ldots,t_j\}$, have common prefix $X_{\e_{j-1}}$. By the symmetry of free products, we get
$$
f(\pi_j)=d_{\mathcal{X}}\bigl(U_0^{(j)},U_{t_j}^{(j)}\bigr) = d_{\mathcal{X}}\Bigl(X_{\e_{j-1}}^{-1}U_0^{(j)},X_{\e_{j-1}}^{-1}U_{t_j}^{(j)}\Bigr)
= f\bigl(X_{\e_{j-1}}^{-1}\pi_j\bigr).
$$
\end{proof}
Since $f\bigl([X_0,\ldots,X_n]\bigr)\leq n$, Condition (\ref{cond: rest bounded}) does hold, and therefore there exists a constant $\mathfrak{c}_{\mathrm{drift}}>0$ such that
$$
\mathfrak{c}_{\mathrm{drift}}=
\lim_{n\to\infty} \frac{d_{\mathcal{X}}(o,X_n)}{n}=
\lim_{n\to\infty} \frac{f\bigl([X_0,\ldots,X_n]\bigr)}{n} \quad \textrm{ almost surely.}
$$
The number is called \textit{drift/rate of escape w.r.t. the graph metric}, and existence (including formulas) was initially shown in \cite{gilch:07}.

\subsubsection{Weights on Edges}
\label{subsub:weight-path-drift}

For $i\in\{1,2\}$, we denote the set of (oriented) edges of $\mathcal{X}_i$ by
$$
\mathcal{E}_i:=\bigl\{ [u,v]\,\bigl|\, u,v\in V_i \textrm{ with }p_i(u,v)>0\bigr\}.
$$
We assign now weights to all  edges in $\mathcal{E}_1\cup \mathcal{E}_2$ by the function
$$
W_0: \mathcal{E}_1\cup \mathcal{E}_2\to\mathbb{R}.
$$
For an edge $e=[u,v]\in \mathcal{E}_1\cup\mathcal{E}_2$, let  $\mathrm{start}(e):=u$ denote the vertex where $e$ starts. In this subsection we make the \textrm{assumption} that, for all sequences $(e_n)_{n\in\N}$ of edges in $\mathcal{E}_1\cup\mathcal{E}_2$ with  $d_{\mathcal{X}}\bigl(o,\mathrm{start}(e_n)\bigr)=n\in\N$, we have
\begin{equation}\label{eq: ass of weights}    
\limsup_{n\to\infty} \frac{\bigl|W_0(e_n)\bigr|}{\exp({\lambda n})} =0 \quad \textrm{ for all $\lambda>0$.}
\end{equation}
That is, the weight of an edge $e\in\mathcal{E}_1\cup\mathcal{E}_2$ may only increase sub-exponentially in the distance of the edge's starting point $\mathrm{start}(e)$ to $o$.
An example, where this assumption is satisfied, is the case $\bigl|W_0(e_n)\bigr|\leq n^d$ for some $d\in\mathbb{N}$.  Also, the assumption is trivially fulfilled if $V_1$ and $V_2$ are finite.
\par
Denote by 
$$
\mathcal{E}:=\bigl\{ [u,v]\,\bigl|\, u,v\in V \textrm{ with } p(u,v)>0\bigr\}
$$
the set of all edges in $\mathcal{X}$. Recall from the construction of $\mathcal{X}$ that each edge in $\mathcal{X}$ arises from exactly one edge in $\mathcal{E}_1\cup \mathcal{E}_2$.
We extend the assignment of weights to all edges in $\mathcal{X}$ by the function
$$
W: \mathcal{E} \to\mathbb{R}, e \mapsto W(e),
$$
where $W(e)=W_0(e_0)$ if the edge $e\in \mathcal{E}$ arises from the edge $e_0\in\mathcal{E}_1\cup \mathcal{E}_2$.
\par
We now define the function
$$
f_W:\Pi\to\mathbb{R}, \pi=[u_0,\ldots,u_m]\mapsto 
\sum_{i=1}^m W\bigl([u_{i-1},u_i]\bigr),
% OLD
%\sum_{e\in \mathcal{E}(\pi)} W(e),
$$
the sum over the weights of all  edges along the  path $\pi$. 

\begin{lemma}\label{lem: range cone add 2}
    The function $f_W$ is cone-additive.
\end{lemma}

\begin{proof}
Note that $\pi_1,\ldots,\pi_{\mathbf{k}(n)}, \pi_n^*$ define a partition of the edges in $[X_0,\ldots,X_n]$, that is,
$$
\mathcal{E}\bigl([X_0,\ldots,X_n]\bigr)=\mathcal{E}(\pi_1) \uplus \mathcal{E}(\pi_2) \uplus\ldots \uplus\mathcal{E}(\pi_{\mathbf{k}(n)})
\uplus \mathcal{E}(\pi_n^*).
$$
For $j,n\in \N$, write $\pi_j=[U_0^{(j)},\ldots,U_{t_j}^{(j)}]$ and $\pi^\ast_n=[U^\ast_{n,0},\ldots,U^\ast_{n,t^\ast}]$.
We now obtain the additivity property as follows:
\begin{eqnarray*}
f_W\bigl([X_0,\ldots,X_n]\bigr)&=&
\sum_{i=1}^n  W\bigl([X_{i-1},X_i]\bigr)\\
&=&
\sum_{j=1}^{\mathbf{k}(n)} \sum_{i=1}^{t_j} W\bigl([U_{i-1}^{(j)},U_i^{(t_j)}]\bigr) +
\sum_{i=1}^{t^\ast} W\bigl([U_{n,i-1}^\ast,U_{n,i}^\ast]\bigr) \\
&=&
\sum_{j=1}^{\mathbf{k}(n)} f_W(\pi_j) + f_W(\pi_n^*).
\end{eqnarray*}
Moreover, by the definition of $W$, 
the weight of an edge $e=[u,v]\in\mathcal{E}$, where $u,v\in C(w)$ for some $w\in V$, does not change if we cancel the common prefix $w$ out of $u$ and $v$, yielding $W(e)=W\bigl([w^{-1}u,w^{-1}v]\bigr)$. Therefore, for all $j\in\N$,
\begin{eqnarray*}
f_W\bigl(X^{-1}_{\e_{j-1}}\pi_j\bigr)&=&\sum_{i=1}^{t_j} W\bigl([X^{-1}_{\e_{j-1}}U_{i-1}^{(j)},X^{-1}_{\e_{j-1}}U_i^{(j)}]\bigr)\\
&=& \sum_{i=1}^{t_j} W\bigl([U_{i-1}^{(j)},U_i^{(j)}]\bigr)=f_W(\pi_j),
\end{eqnarray*}
which proves shift invariance.
\end{proof}
We obtain the next new limit theorem:

\begin{thm}\label{thm: weights}  
Assume that $\mathscr{R}>1$. Let $W$ be a weight function as above such that (\ref{eq: ass of weights}) is satisfied.
Then there exists a constant $\mathfrak{c}_{\textrm{weight}}$ such that
$$
\mathfrak{c}_{\textrm{weight}}=\lim_{n\to\infty} \frac{f_W\bigl([X_0,\ldots,X_n]\bigr)}{n} \quad \textrm{ almost surely.}
$$
\end{thm}

\begin{proof}
We shall verify that Condition (\ref{cond: rest bounded}) of Theorem~\ref{thm:limit-theorem} is satisfied.
By Assumption (\ref{eq: ass of weights}) we get that, for every $\lambda>0$, there exists a constant $C_{\lambda}>0$  such that 
$$ 
\vert W(e)\vert \leq C_{\lambda}\cdot \exp(\lambda \cdot n) \quad \textrm{ for all $e\in \mathcal{E}$ with $d_{\mathcal{X}}\bigl(o,\mathrm{start}(e)\bigr)=n$.}
$$
Thus, for all $\lambda>0$ and $\pi=[u_0,\ldots,u_m]\in\Pi$,
\begin{eqnarray*}
\bigl| f_W(\pi)\bigr| &\leq & \sum_{i=1}^m \bigl|W([u_{i-1},u_i])\bigr| \leq C_\lambda \cdot \mathrm{length}(\pi)\cdot \exp\bigl(\lambda \cdot \mathrm{length}(\pi)\bigr) 
\end{eqnarray*}
Choosing $\lambda>0$ sufficiently small guarantees now that Condition (\ref{cond: rest bounded}) holds. Together with cone-additivity of $f_W$ we obtain the proposed theorem. 
\end{proof}
The constant $\mathfrak{c}_{\textrm{weight}}$ can be regarded as the average cost per single random walk step.

\subsubsection{Distances on Edges}

Consider again the weight assignment function $W:\mathcal{E}\to\mathbb{R}$ from Subsection \ref{subsub:weight-path-drift}, where we now assume that $W(e)\geq 0$ for all $e\in\mathcal{E}$. Weights are now interpreted as distances. For $u,v\in V$, set
$$
d_{\mathrm{weights}}(u,v):= \inf\bigl\{f_W(\pi) \,\bigl|\, \pi \textrm{ is a path from $u$ to $v$}\bigr\},
$$
the minimal distance/weight of a path from $u$ to $v$.
For $\pi\in\Pi$, define the function
$$
\ell_W:\Pi \to [0,\infty), [u_0,\ldots,u_m] \mapsto d_{\mathrm{weights}}(u_0,u_m).
$$
\begin{lemma}\label{lem:ellW-cone-additive}  
    $\ell_W$ is cone-additive.
\end{lemma}
\begin{proof}
Let be $n\in\N$ and consider the random path $[X_0,\ldots,X_n]$. Then a path from $X_0=o$ to $X_n$ has to pass through $X_{\e_1},\ldots,X_{\e_{\kk(n)}}$. Therefore, we must have
\begin{eqnarray*}
\ell_W(\pi) &=& \sum_{j=1}^{\kk(n)} d_{\mathrm{weights}}\bigl(X_{\e_{j-1}},X_{\e_j}\bigr) + d_{\mathrm{weights}}\bigl(X_{\e_{\kk(n)}},X_n\bigr) \\
&=& \sum_{j=1}^{\kk(n)} \ell_W(\pi_j) + \ell_W(\pi_n^\ast),
\end{eqnarray*}
since $\pi_j$, $j\in\N$, is a path from $X_{\e_{j-1}}$ to $X_{\e_{j}}$ and $\pi_n^\ast$ is a path from $X_{\e_{\kk(n)}}$ to $X_n$. This proves cone-additivity.
\par
For the proof of shift-invariance, we remark that, due to the structure of the free product and the definition of $W_0$, we have for every $j\in\N$
$$
\ell_W(\pi_j)=d_{\mathrm{weights}}\bigl(X_{\e_{j-1}},X_{\e_j}\bigr) 
= d_{\mathrm{weights}}\bigl(o,X_{\e_{j-1}}^{-1}X_{\e_j}\bigr)
=\ell_W\bigl(X_{\e_{j-1}}^{-1}\pi_j\bigr).
$$
This proves shift invariance; hence, $\ell_W$ is cone-additive.
\end{proof}
Now we obtain the following generalisation of a result in \cite{gilch:07}:
\begin{thm}\label{thm:weight-distance}
   Assume that $\mathscr{R}>1$. Let $W$ be a non-negative weight function as above such that (\ref{eq: ass of weights}) is satisfied.
Then there exists a constant $\mathfrak{c}_{\textrm{weight-dist}}$ such that
$$
\mathfrak{c}_{\textrm{weight-dist}}=\lim_{n\to\infty} \frac{\ell_W\bigl([X_0,\ldots,X_n]\bigr)}{n} \quad \textrm{ almost surely.}
$$ 
\end{thm}
\begin{proof}
In view of Lemma \ref{lem:ellW-cone-additive}, it remains to show that Condition (\ref{cond: rest bounded}) in Theorem \ref{thm:limit-theorem} holds. But this follows from the proof of Theorem \ref{thm: weights} since, for every $\pi\in\Pi$,
$$
\ell_W(\pi) \leq f_W(\pi).
$$
\end{proof}
We remark that the above theorem does not assume that the weights on the edges are bounded. In the special case of  $V_1$ and $V_2$ being finite or if $W_0$ is bounded, the above theorem was proven by Gilch \cite{gilch:07}. Hence,  Theorem \ref{thm:weight-distance} generalises that result.

\subsection{Lamplighter Distance/Travelling Salesman}

We consider a lamp\-lighter random walk on the free product $V$ where a lamp sits at each vertex of $V$, which can have the states ``on'' (or $1$) or ``off'' (or $0$). A lamplighter performs now a random walk on $V$, where the lamplighter can switch on or off the lamp at the currently visited vertex.
\par
For this purpose, consider the set of all finitely supported lamp configurations on the vertices given by 
$$
\conf:=\bigl\{\phi: V\longrightarrow \{0,1\} \,\bigl|\,  \vert \supp(\phi)\vert <\infty\bigr\}, 
$$
where $\phi(x)=1$ means that the lamp at $x\in V$ is on and $\phi(x)=0$ means that the lamp at $x$ is off. 
\par
Let be $\beta\in (0,1)$. 
The \textit{lamplighter random walk} is now given by the sequence of random variables
$$
\bigl((X_n,\Phi_n)\bigr)_{n\in\N_0},
$$
where $X_n$ represents the lamplighter's position  and $\Phi_n$ the random lamp configuration at time $n$, which has
the following single-step transition probabilities on the state space $V\times\conf$:
$$
p_{\mathrm{LL}}\bigl((x_1,\phi_1),(x_2,\phi_2)\bigr):= \begin{cases}
p(x_1,x_2)\cdot \beta, & \textrm{if } \mathrm{supp}(\phi_2-\phi_{1})=\{x_2\},\\
p(x_1,x_2)\cdot (1-\beta), & \textrm{if } \phi_1=\phi_2,\\
0, & \textrm{otherwise,}
\end{cases}
$$
where $(x_1,\phi_1),(x_2,\phi_2)\in V\times \conf$.
In other words, the lamplighter random walk can walk in one step from $x_1$ to $x_2$ with probability $p(x_1,x_2)$ and at $x_2$ the lamplighter may change the lamp state with probability $\beta$ or keep the lamp state there. These two acts are seen as one single step in the random walk process. Note that lamp changes are only allowed when arriving at some site. This random walk is also called \textit{walk-switch lamplighter random walk}. Initially, we set $X_0:=o$ and $\Phi:=\mathbf{0}$, the constant zero function on $V$. We refer to \cite[pp. 72]{gilch:diss} for more details in the case of lamplighter random walks on free products of groups.
\par
For $\pi=[u_0,\ldots,u_n]\in\Pi$, $n\in\N$, define
$$
\mathrm{Conf}(\pi):=\bigl\{ \phi:V\to  \{0,1\} \,\bigl|\, \mathrm{supp}(\phi)\subseteq \{u_0,\ldots,u_n\}\bigr\},
$$ 
the set of all possible configurations along the path $\pi$.
Set 
$$
\mathcal{U}:=\bigl\{ (\pi,\phi) \,\bigl|\, \pi\in\Pi, \phi\in \mathrm{Conf}(\pi)\bigr\}.
$$
For $(x,\phi)\in\mathcal{U}$,
we are interested in the length of a shortest path from $o$ to $x$ which visits all vertices whose lamps are ``on" w.r.t. $\phi$, that is, a shortest path from $o$ to $x$ which visits all elements in $\mathrm{supp}(\phi)$. This can be regarded as a \textit{travelling salesman problem}.
For this purpose, we define the lamplighter distance
$$
\ell_{\mathrm{LL}}: \mathcal{U} \to \N_0, (\pi,\phi)\mapsto \ell_{\mathrm{LL}}(\pi,\phi),
$$
where, for $\pi=[u_0,\ldots,u_n]\in\Pi$ and $\varphi\in \mathrm{Conf}(\pi)$,
$$
\ell_{\mathrm{LL}}(\pi, \phi):= \min\left\{\mathrm{length}(\widetilde{\pi}) \,\Biggl| \, \begin{array}{c}
\widetilde\pi\in\Pi \textrm{ is a path from $u_0$ to $u_n$},\\
\textrm{ which visits all elements in $\supp(\phi)$}
\end{array}
\right\}.
$$
Furthermore, for $v\in V\setminus\{o\}$, $\pi=[u_0,\ldots,u_n]\in \Pi$ with $\{u_0,\ldots,u_n\}\subseteq C(v)$ and $\phi\in \mathrm{Conf}(\pi)$, we define the shifted configuration
$$
v^{-1}\phi: V\to \{0,1\}, x \mapsto \bigl(v^{-1}\phi\bigr)(x):=\begin{cases}
1, & \textrm{if $vx$  well-defined $\&$ }  \phi(vx)=1,\\
0, & \textrm{otherwise}.
\end{cases}
$$
That is, $v^{-1}\phi$ arises from $\phi$ by deleting the vertex $v$ at each $w\in V$ where the lamp is on.
Analogously to the path pieces $(\pi_j)_{j\in\N}$ in Section~\ref{subsec: paths}, we consider the corresponding lamp configuration along these pieces, with the only difference that at the end vertex of $\pi_j$ we may omit the lamp state, since it does not affect the value of the function $\ell_{\mathrm{LL}}$. Recall that $\Phi_n$ denotes the random lamp configuration at time $n$.
For $j\in\N$, we define 
$$
\eta_j := \Phi_{\e_j}\cdot \mathds{1}_{C(X_{\e_{j-1}})\setminus C(X_{\e_j})}
$$
and for $n\in\N$,
$$
\eta_n^\ast := \Phi_{n}\cdot \mathds{1}_{C\bigl(X_{\e_{\kk(n)}}\bigr)}.
$$
Observe that the lamp configurations in $C(X_{\e_{j-1}})\setminus C(X_{\e_j})$ do not change any more after time $\e_j$. With this notation, we have
$$
\mathrm{supp}(\eta_1)\uplus \mathrm{supp}(\eta_2)\uplus \ldots \uplus \mathrm{supp}(\eta_{\kk(n)})\uplus \mathrm{supp}(\eta_n^\ast)= \mathrm{supp}(\Phi_n).
$$
The following lemma shows that $\ell_{\mathrm{LL}}$ satisfies some generalized cone-additivity properties.
\begin{lemma}\label{lem: LL gen add}
  % The function $f$ satisfies the following properties:
    \begin{enumerate}[label=(\roman*)]
    \item For all $n\in\N$, we have
    $$
    \ell_{\mathrm{LL}}\bigl([X_0,\ldots, X_n],\Phi_n\bigr)
= \sum_{j=1}^{\kk(n)} \ell_{\mathrm{LL}}\bigl(\pi_j,\eta_j\bigr)+ \ell_{\mathrm{LL}}\bigl(\pi^\ast_n,\eta^\ast_n\bigr)\quad \textrm{ almost surely.}
    $$
\item For all $v\in V\setminus\{o\}$, $\pi=[u_0,\ldots,u_n]\in \Pi$ with $\{u_0,\ldots,u_n\}\subseteq C(v)$ and $\phi\in \mathrm{Conf}(\pi)$, 
    $$
\ell_{\mathrm{LL}}\bigl(\pi,\phi\bigr)=\ell_{\mathrm{LL}}\bigl(v^{-1}\pi,v^{-1}\phi\bigr).
    $$ 
    \end{enumerate}
\end{lemma}

\begin{proof}
The value $f\bigl([X_0,\ldots, X_n],\Phi_n\bigr)$ only depends on $X_0, X_n$ and $\Phi_n$. Any path from $X_0$ to $X_n$ has to pass through $X_{\e_1},\ldots,X_{\e_{\kk(n)}}$.
Moreover, a shortest path from $X_{\e_{j-1}}$ to $X_{\e_j}$ visiting all lamps in $\mathrm{supp}(\eta_j)$ will stay in the sphere $\mathcal{S}_j$, because any detour to another sphere would not give rise to a shorter path. Vice versa, each element $x\in\mathrm{supp}(\eta_i)$, $i\neq j$, can be taken into account by considering the corresponding shortest path from $X_{\e_{i-1}}$ to $X_{\e_i}$ visiting $x$.
\par
Therefore, since the length of a path is additive, we obtain $(i)$ by
\begin{eqnarray*}
&&\ell_{\mathrm{LL}}\bigl([X_0,\ldots,X_n],\Phi_n\bigr)\\
&= & \sum_{j=1}^{\kk(n)} \ell_{\mathrm{LL}}\Bigl(\pi_j, \Phi_{\e_j} \cdot  \mathds{1}_{C(X_{\e_{j-1}})\setminus C(X_{\e_j})}\Bigr)
 +
 \ell_{\mathrm{LL}}\Bigl(\pi_n^\ast, \Phi_n \cdot  \mathds{1}_{C(X_{\e_{\kk(n)}})}\Bigr)\\
 &=& \sum_{j=1}^{\kk(n)} \ell_{\mathrm{LL}}\bigl(\pi_j, \eta_j\bigr)
 +
 \ell_{\mathrm{LL}}\bigl(\pi_n^\ast, \eta_n^\ast\bigr).
\end{eqnarray*}
Property $(ii)$ follows by symmetry of the graph. The configuration does not change and is just pulled back to the corresponding path shifted by $v^{-1}$. The length of a shortest path from $v_0$ to $v_n$ visiting all elements in $\supp(\phi)$ is the same as the length of a shortest path from $v^{-1}v_0$ to $v^{-1}v_n$ visiting all elements in $\supp(v^{-1}\phi)$.
\end{proof}
Now we can prove that the rate of escape of the lamplighter random walk exists and is strictly positive:
\begin{thm}\label{thm:LL}
There exists a real number $\mathfrak{c}_{LL}\in(0,\infty)$ such that
$$
\mathfrak{c}_{LL}=\lim_{n\to\infty} \frac1n \ell_{\mathrm{LL}}\bigl([X_0,\ldots,X_n],\Phi_n\bigr) \quad \textrm{almost surely.}
$$
\end{thm}
\begin{proof}
Define
$$
\zeta_1:V\to \{0,1\}, x \mapsto \Phi_{\e_1}(x)\cdot \mathds{1}_{C_{\delta(X_{\e_1})}}(x).
$$
For $k\geq 2$, define
$$
\zeta_k:V\to \{0,1\}, x \mapsto \Bigl(X_{\e_{k-1}}^{-1}\bigl(\Phi_{\e_k}\cdot \mathds{1}_{C(X_{\e_{k-1}})}\bigr)\Bigr)(x),
$$
that is, $\zeta_k(x)=1$ if and only if $X_{\e_{k-1}}x$ is well-defined and $\Phi_{\e_k}(X_{\e_{k-1}}x)=1$.
With this notation, we have for all $k\geq 2$
$$
X_{\e_{k-1}}^{-1}\eta_k=\zeta_k \cdot \mathds{1}_{C(\W_k)^c}
$$
and
\begin{equation}\label{equ:ellLL-decomposition}
\ell_{\mathrm{LL}}(\pi_k,\eta_k)=\ell_{\mathrm{LL}}\bigl( X_{\e_{k-1}}^{-1} \pi_k,X_{\e_{k-1}}^{-1}\eta_k\bigr)
= \ell_{\mathrm{LL}}\Bigl(\psi_k\bigl|_{\neg C(\W_k)},\zeta_k\cdot \mathds{1}_{C(\W_k)^c}\Bigr).
\end{equation}
Now one can prove completely analogously to Lemma \ref{lem:support-D} and Propositions \ref{prop: irred MC}, \ref{prop:positive-recurrence} 
that $\bigl((\W_k,\psi_k,\zeta_k)\bigr)_{k\in\N}$ forms a homogeneous, irreducible, positive-recurrent Markov chain on the state space $\mathcal{D}_{\mathrm{LL}}:=\mathrm{\supp}\bigl((\W_1,\psi_1,\zeta_1)\bigr)$ with invariant distribution $\varrho_{LL}$. Analogously to the proof of Proposition \ref{prop:positive-recurrence} one can show that $(g,[0,g],\mathbf{0})$, $g\in V_1$, is positive-recurrent, where $\mathbf{0}$ represents the zero function.
\par
Equation (\ref{equ:f-decomposition}) becomes in this setting
\begin{eqnarray*}
&&\ell_{\mathrm{LL}}\bigl([X_0,\ldots,X_n],\Phi_n\bigr)\\
&=& \ell_{\mathrm{LL}}(\pi_1,\eta_1) + \sum_{i=2}^{\kk(n)} \ell_{\mathrm{LL}}\Bigl(\psi_i\bigl|_{\neg C(\W_i)},\zeta_i\cdot \mathds{1}_{C(\W_i)^c}\Bigr) + \ell_{\mathrm{LL}}\bigl(\pi^\ast_n,\eta_n^\ast\bigr),
\end{eqnarray*}
by Lemma~\ref{lem: LL gen add} and Equation (\ref{equ:ellLL-decomposition}).
Completely analogously to Lemmas \ref{lem: beginning to 0} and \ref{lem: end to 0}, one can show that 
$$
\frac1n{}\ell_{\mathrm{LL}}(\pi_1,\eta_1)\to 0, \quad \frac1n  \ell_{\mathrm{LL}}\bigl(\pi^\ast_n,\eta_n^\ast\bigr)\to 0 \quad \textrm{ almost surely,}
$$
since $\ell_{\mathrm{LL}}(\pi,\phi)\leq \mathrm{length}(\pi)$ for all $(\pi,\phi)\in\mathcal{U}$, which in turn implies existence of $C\in(0,\infty)$ and $R\in(0,\mathcal{R})$ such that
$$
\bigl|\ell_{\mathrm{LL}}(\pi,\phi)\bigr| \leq C\cdot R^{\mathrm{length}(\pi)} \quad \textrm{ for all } (\pi,\phi)\in\mathcal{U}.
$$
This yields an analogous version of Corollary \ref{cor:expectation-psi2-restricted-finite}:
\begin{eqnarray*}
&&\mathbb{E}_{\varrho_{LL}}\Bigl[\ell_{\mathrm{LL}}\bigl(\psi_1\bigl|_{\neg C(\W_1)},\zeta_1\cdot \mathds{1}_{C(\W_1)^c}\bigr)\Bigr]\\
&:=&
\sum_{(g,\pi,\phi)\in\mathcal{D}_{\mathrm{LL}}} \varrho_{LL}(g,\pi,\phi) \cdot \ell_{\mathrm{LL}}\Bigl(\pi\bigl|_{\neg C(g)},\phi\cdot \mathds{1}_{C(g)^c}\Bigr) <\infty.
\end{eqnarray*}
Once again, completely analogously to the proof of Lemma \ref{lem: to E[f(pi)]} one can show that
\begin{eqnarray*}
&&\lim_{n\to\infty} \frac1{\kk(n)}\sum_{i=2}^{\kk(n)}\ell_{\mathrm{LL}}\Bigl(\psi_i\bigl|_{\neg C(\W_i)},\zeta_i\cdot \mathds{1}_{C(\W_i)^c}\Bigr) \\
&=&
\mathbb{E}_{\varrho_{LL}}\Bigl[\ell_{\mathrm{LL}}\Bigl(\psi_1\bigl|_{\neg C(\W_1)},\zeta_1\cdot \mathds{1}_{C(\W_1)^c}\Bigr)\Bigr] \quad \textrm{almost surely,}
\end{eqnarray*}
from which we finally obtain the almost sure convergence 
$$
\lim_{n\to\infty} \frac1n \ell_{\mathrm{LL}}\bigl([X_0,\ldots,X_n],\Phi_n\bigr) =
\frac{\mathbb{E}_{\varrho_{LL}}\Bigl[\ell_{\mathrm{LL}}\Bigl(\psi_1\bigl|_{\neg C(\W_1)},\zeta_1\cdot \mathds{1}_{C(\W_1)^c}\Bigr)\Bigr]}{\mathbb{E}_\sigma[\e_2-\e_1]}.
$$
This proves the claim.
\end{proof}

\iffalse
\subsection{Lamplighter on edges}
In order to apply Theorem~\ref{thm:limit-theorem} to the above-defined function $f$, we need to define a graph structure, which includes the states of the lamps. We shall do it as follows. Note that this is not quite a model for the lamplighter, but it captures all the information we need in order to read off the values of the function $f$.
We "blow up" $V_i$ by doubling all vertices besides the root: 
\begin{center}$x\in V_i^\times$ gives $x_{\on}$ and $x_{\off}$ in our new graph $\widetilde{V_i}$.
 \end{center}
If there is an edge $[x,y]$ in $V_i$, then there shall be edges $[x_{\on},y_{\off}]$, $[x_{\on},y_{\on}]$, $[x_{\off},y_{\off}]$, $[x_{\off},y_{\on}]$ in $\widetilde{V_i}$. If $x=\id_i$, then the above shall hold for $x=x_{\off}=x_{\on}$ and identifying double edges. The same if $y=\id_i$.
Now, we consider $$\widetilde{V}:=\widetilde{V_1}*\widetilde{V_2}.$$
We interpret a path $[x_1\ldots x_n]$ in $V$ with configuration $f_i\in\mathbb{Z}/2\mathbb{Z}$ above $x_i$ as follows:
\fi

\end{document}